\documentclass[a4paper]{article}

\newcommand{\COLORON}{1}
\newcommand{\NOTESON}{0}
\newcommand{\Debug}{0} 

\usepackage{amsthm,amssymb,amsmath,enumerate,graphicx,epsf}
\usepackage[numbers,sort&compress]{natbib}
\usepackage[normalem]{ulem}
\usepackage{bbold}
\usepackage{mathtools,enumitem}
\usepackage{authblk}
\usepackage[percent]{overpic}
\usepackage{hyperref}
\RequirePackage{latexsym} \RequirePackage{amsthm}
\RequirePackage{amsmath}
\usepackage{authblk}
\usepackage{tikz}
\usepackage[dvipsnames]{xcolor}
\usepackage{subcaption}
\usepackage{geometry}
\usetikzlibrary{arrows.meta, positioning, shapes.geometric}
\usetikzlibrary{positioning, calc, shapes.misc, fit}

\usepackage{enumerate}
\RequirePackage{amssymb} \RequirePackage{makeidx}
\usepackage{dsfont}
\usepackage{mathrsfs}
\usepackage[nameinlink]{cleveref}
\usepackage[T1]{fontenc}

\usepackage{amsthm,amssymb,amsmath,bbm,enumerate,graphicx,epsf,stmaryrd,accents}

\usepackage{authblk}

\newcommand{\comment}[1]{}
\newcommand{\COMMENT}[1]{}

\definecolor{darkgray}{rgb}{0.3,0.3,0.3}
\newcommand{\defi}[1]{{\color{darkgray}\emph{#1}}}

\comment{
	\begin{lemma}\label{}	
\end{lemma}
\begin{proof}

\end{proof}

\begin{theorem}\label{}
\end{theorem} 
\begin{proof} 	

\end{proof}

}

\newtheorem{proposition}{Proposition}[section]
\newtheorem{definition}[proposition]{Definition}
\newtheorem{theorem}[proposition]{Theorem}
\newtheorem{corollary}[proposition]{Corollary}
\newtheorem*{nocorollary}{Corollary}
\newtheorem{lemma}[proposition]{Lemma}
\newtheorem{observation}[proposition]{Observation}
\newtheorem{conjecture}{{Conjecture}}[section]

\newtheorem{problem}[conjecture]{{Problem}}

\newtheorem{question}[conjecture]{{Question}}

\newtheorem{examp}[proposition]{Example}

\theoremstyle{definition}
\newtheorem{remark}[proposition]{Remark}

\theoremstyle{definition}

\newcommand{\kreis}[1]{\mathaccent"7017\relax #1}

\newcommand{\FIG}{0}

\ifnum \NOTESON = 1 \newcommand{\note}[1]{ 

\hspace*{-30pt}
	{\color{blue}  NOTE: \color{Turquoise}{\small  \tt \begin{minipage}[c]{1.1\textwidth}  #1 \end{minipage} \ignorespacesafterend }} 
	
	}
\else \newcommand{\note}[1]{} \fi

\newcommand{\afsubm}[1]{ \ifnum \Debug = 1 {\mymargin{#1}}
\fi} 

\ifnum \Debug = 1 
\else  \fi

\ifnum \FIG = 1 \newcommand{\fig}[1]{Figure ``{#1}''}
\else \newcommand{\fig}[1]{Figure~\ref{#1}} \fi

\ifnum \FIG = 1 
\else  \fi

\ifnum \Debug = 1 \usepackage[notref,notcite]{showkeys}
\fi

\ifnum \COLORON = 0 \renewcommand{\color}[1]{}
\fi

\newcommand{\N}{\ensuremath{\mathbb N}}
\newcommand{\R}{\ensuremath{\mathbb R}}

\newcommand{\Z}{\ensuremath{\mathbb Z}}

\newcommand{\cc}{\ensuremath{\mathcal C}}

\newcommand{\ci}{\ensuremath{\mathcal I}}

\newcommand{\cp}{\ensuremath{\mathcal P}}

\newcommand{\cs}{\ensuremath{\mathcal S}}

\newcommand{\cx}{\ensuremath{\mathcal X}}

\newcommand{\eps}{\ensuremath{\epsilon}}

\newcommand{\sm}{\backslash}

\newcommand{\isom}{\cong}

\makeatletter
\DeclareRobustCommand{\cev}[1]{%
  \mathpalette\do@cev{#1}%
}
\newcommand{\do@cev}[2]{%
  \fix@cev{#1}{+}%
  \reflectbox{$\m@th#1\vec{\reflectbox{$\fix@cev{#1}{-}\m@th#1#2\fix@cev{#1}{+}$}}$}%
  \fix@cev{#1}{-}%
}
\newcommand{\fix@cev}[2]{%
  \ifx#1\displaystyle
    \mkern#23mu
  \else
    \ifx#1\textstyle
      \mkern#23mu
    \else
      \ifx#1\scriptstyle
        \mkern#22mu
      \else
        \mkern#22mu
      \fi
    \fi
  \fi
}

\makeatother

\newcommand{\pth}[2]{\ensuremath{#1}\text{--}\ensuremath{#2}~path}

\newcommand{\g}{\ensuremath{G\ }}
\newcommand{\G}{\ensuremath{G}}

\newcommand{\Cg}{Cayley graph}

\newcommand{\Lr}[1]{Lemma~\ref{#1}}

\newcommand{\Tr}[1]{Theorem~\ref{#1}}
\newcommand{\Trs}[1]{Theorems~\ref{#1}}
\newcommand{\Sr}[1]{Section~\ref{#1}}

\newcommand{\Prr}[1]{Pro\-position~\ref{#1}}
\newcommand{\Prb}[1]{Problem~\ref{#1}}
\newcommand{\Cr}[1]{Corollary~\ref{#1}}
\newcommand{\Cnr}[1]{Con\-jecture~\ref{#1}}
\newcommand{\Or}[1]{Observation~\ref{#1}}

\newcommand{\Qr}[1]{Question~\ref{#1}}
\newcommand{\Rr}[1]{Remark~\ref{#1}}

\renewcommand{\iff}{if and only if}
\newcommand{\fe}{for every}

\newcommand{\st}{such that}

\newcommand{\obda}{without loss of generality}

\newcommand{\wrt}{with respect to}

\newcommand{\labequ}[2]{ \begin{equation} \label{#1} #2 \end{equation} } 
 
\newcommand{\labtequ}[2]{
 \begin{equation} \label{#1} 	\begin{minipage}[c]{0.9\textwidth}  #2 \end{minipage} \ignorespacesafterend \end{equation} }

\newcommand{\mymargin}[1]{
 \ifnum \Debug = 1
  \marginpar{%
    \begin{minipage}{\marginparwidth}\small%
      \begin{flushleft}%
        {\color{blue}#1}%
      \end{flushleft}%
   \end{minipage}%
  }%
 \fi
}%

\newcommand{\extras}[1]{
 \ifnum \Debug = 1
\section{Extras} #1
 \fi
}%

\newcommand{\mySection}[2]{}

\newcommand{\cococ}{coarsely contraction closed}

\newcommand{\Id}{\ensuremath{\mathrm{Id}}}

\newcommand{\Hd}{\ensuremath{d^{HD}}}

\newcommand{\rig}{region-intersection-graph}
\newcommand{\ing}{intersection graph}
\newcommand{\cig}{cover-intersection-graph}
\newcommand{\scig}{subdivision-cover-intersection-graph}
\newcommand{\CIG}{\ensuremath{\mathcal{CIG}}}
\newcommand{\SCIG}{\ensuremath{\mathcal{SCIG}}}

\newcommand{\QI}{\ensuremath{\mathcal{QI}}}

\newcommand{\Forb}{\mathrm{Forb}}
\newcommand{\Find}{\Forb_{\mathrm{ind}}}

\newcommand{\qy}{quasi-isometry}
\newcommand{\qic}{quasi-isometric}

\begin{document}

\title{Quasi-isometries, contractions, and intersection graphs}

\author[1]{Agelos Georgakopoulos\thanks{Supported by EPSRC grant  EP/V009044/1. AG thanks the department of mathematics of the NKUA for hospitality
while some of this work was carried out.}}
\author[2]{Chiara Molinari\thanks{Supported by the Warwick Mathematics Institute CDT and 
acknowledges funding from the University of Warwick.}}

\affil[1,2]{  {Mathematics Institute}\\ {University of Warwick}\\  {CV4 7AL, UK}}

\date{\today}
\maketitle

\begin{abstract}
We prove that a graph $G$ is quasi-planar ---i.e.\ quasi-isometric to a planar graph--- if and only if it can be obtained by iterating the following two operations a bounded number of times: a) subdividing each edge into a path of bounded length, and b) taking the intersection graph of a family of connected subgraphs covering $G$. This applies  both to infinite graphs, and to families of finite graphs with uniform constants.

The backward implication relies on, and generalises, a deep result of Davies, partly proved independently by Chang, Conroy, Tan \& Zheng, saying that every string graph is quasi-planar. The forward implication requires new ideas. 

As a byproduct of our proofs, we  deduce that every contraction minor of a quasi-planar graph is quasi-planar. Moreover, if \g admits a tree-decomposition with adhesions of bounded diameter and quasi-planar induced bags, then \g is itself quasi-planar.

Our results apply to other graph classes as well, and we offer various tools for understanding quasi-isometries as well as bi-Lipschitz equivalences between graphs. 
\end{abstract}

intersection graph, string graph, contraction. \\ 


\maketitle

\section{Introduction}



There is widespread interest in studying graph properties that are invariant under quasi-isometries, as such properties are stable under local perturbations and capture the large-scale geometric structure. This parallels ideas in Geometric Group Theory following Gromov's paradigm \cite{GroAsyInv}, whereby quasi-isometries between groups play a central role. More broadly, there is interest in obtaining metric variants of graph-theoretic results. See e.g.\ \cite{ADGSmallCounterexamples,AlbGeoBlo,AJKW,BerSeyBou,BlPiPrCoa,BBEGLPS,BLPPSep,DGHLMGallai,FujPapCoa,FujPapAsy,LiuWeakCoarseMenger,McMFat,NgScSeAsyI,NgScSeAsyII} for a sample of the fast-growing literature in this direction, and \cite{GeoPapMin} for a discussion of the motivation. 

Of particular interest, both from the graph-theoretic \cite{BlPiPrCoa, CCTZstring,DaviesStringGraphsQuasiPlanar,NgScSeAsyVI}, as well as the group-theoretic  perspective \cite{EsGiCoa,McMAcc,McMFat,MaMaSpAlm}, is the class of \defi{quasi-planar} graphs, i.e.\  graphs that are $(M,A)$-quasi-isometric to a planar graph for fixed constants $M,A$. Well known-open problems ask whether such graphs admit bounded-distortion embeddings into $\ell_1$ \cite{GNRS}, whether $M$ can always be chosen to be 1 by changing the target planar graph \cite{NgScSeAsyII,GeoVigCoa}, and whether quasi-planarity is equivalent to forbidding the Kuratowski graphs as asymptotic minors \cite{GeoPapMin}.

This paper revolves around quasi-planar graphs, and provides a purely combinatorial characterization involving intersection graphs. Along the way, we will develop tools for studying quasi-isometry and bi-Lipschitz equivalence in much broader setups, some of which are of independent interest. We will now review some of these results and techniques, leading up to our aforementioned characterization.

\medskip
Our starting point is the following deep result, recently proved by Davies \cite{DaviesStringGraphsQuasiPlanar} and independently by Chang et al \cite{CCTZstring}.\footnote{Chang et al \cite{CCTZstring} only consider finite graphs, but Davies' proof applies to countably infinite ones.} 

\begin{theorem} [{\cite{CCTZstring,DaviesStringGraphsQuasiPlanar}}] \label{Davies}
Every countable string graph is $(M,A)$-\qic\ to a planar graph for universal constants $M,A$.\footnote{For our purposes, a \defi{string graph} can be defined as an intersection of a family of connected subgraphs of a planar graph. The definition of string graph in {\cite{CCTZstring,DaviesStringGraphsQuasiPlanar}} is more general, but it is easy to see that ours is a special case.}
\end{theorem}

Our first result uses \Tr{Davies} to prove the following statement, conjectured by the first author in recent meetings. We say that $G$ is a \defi{contraction minor} of a graph $H$, if \g is obtained from $H$ by contracting edges.

\begin{theorem} \label{cm to qp}
    Let $G$ be a graph quasi-isometric to a planar graph. Then every contraction minor of $G$ is quasi-isometric to a planar graph.
\end{theorem}

As is usual in the area, all our results come in two variants: one where \g is a (countably) infinite graph, and one where \g belongs to an infinite family of finite graphs. The same proofs apply to both setups. In this introduction we chose to present the infinite versions, as they are simpler to state. But all our proofs are constructive and provide explicit bounds and efficient algorithms for families of finite graphs that will be stated throughout the paper. For example, \Tr{cm to qp} provides an explicit function $f: \N^2 \to \N^2$ such that if \g is $(M,A)$-quasi-planar, then every contraction minor of \g is $f(M,A)$-quasi-planar. 

\medskip
\Tr{cm to qp} follows from the following more general statement. 
Recall that a \defi{\rig} of a graph $G$ is the intersection graph\footnote{See \Sr{sec gra} for the definition.}  \wrt\ a family $\{S_i\}_{i\in \ci}$ of connected subgraphs of $G$. When, in addition, $\bigcup_{i\in \ci} S_i = G$, we call $S$ a \defi{\cig} of $G$. (The benefit of \cig s compared to \rig s is that they respect the geometry of the underlying graph, much like induced subgraphs respect more of the structure than arbitrary subgraphs.)

\begin{theorem} \label{cm to qp gen}
Let $G$ be a graph quasi-isometric to a graph $H$. Then every contraction minor of $G$ is quasi-isometric to a \cig\ of $H$ (in particular, to a \cig\ of $G$ itself).
\end{theorem}

To see that this implies \Tr{cm to qp}, note that if $H$ is planar, then any \cig\ of $H$ is a string graph by definition, and therefore quasi-planar by \Tr{Davies}.


The following is an approximate converse of \Tr{cm to qp gen}:

\begin{theorem} \label{cig to cm}
Let $S$ be a \cig\ over a graph $G$. Then $S$ is isomorphic to a contraction minor of a graph $G'$ which is $(2,1)$-quasi-isometric to \G.
\end{theorem}

We proved \Tr{cm to qp} using \Tr{Davies}. Conversely, any proof of \Tr{cm to qp} would imply, via \Tr{cig to cm}, a proof of \Tr{Davies}.

\comment{
As an immediate corollary, we deduce that every string graph is a contraction minor of a graph $(2,1)$-quasi-isometric to a planar graph.
}

By combining \Tr{cm to qp gen} and \Tr{cig to cm}, we deduce the following corollaries: 

\begin{corollary} \label{cor cigs}
    Let $G,H$ be quasi-isometric graphs. Then every \cig\ $S$ of \g is quasi-isometric to a \cig\ of $H$.
\end{corollary}
\begin{proof}
By \Tr{cig to cm}, $S$ is a contraction minor of a graph $G'$ quasi-isometric to $G$, and therefore to $H$. Applying \Tr{cm to qp gen} with $G$ replaced by $G'$ yields that $S$ is quasi-isometric to a \cig\ of $H$.
\end{proof}

\Tr{Davies} provides interesting examples of quasi-planar graphs, but it cannot  capture all $(M,A)$-quasi-planar graphs just because there are such graphs with arbitrarily large $M,A$, while the constants in \Tr{Davies} are universal. However, we will come closer to capturing all quasi-planar graphs by iterating the operation of taking a \cig\ implicit in the definition of a string graph as follows. Given a graph ---or a class of graphs--- $H$, we start by letting $\mathcal{CIG}^0(H) = H$, and recursively define $\mathcal{CIG}^n(H)$ for every $n\in \N_{>0}$ to be the class of \cig s of a graph in $\mathcal{CIG}^{n-1}(H)$. A graph \g is an \defi{iterated cover-intersection graph} of $H$ if $G \in \mathcal{CIG}^n(H)$ for some $n$. When $H$ is planar, we also call such a \g an \defi{iterated string graph}. We prove

\begin{corollary} \label{cor isg}
    Every iterated string graph 
    is quasi-planar.
\end{corollary}
\begin{proof}
Suppose $S\in \mathcal{CIG}^n(\cp)$, where \cp\ stands for the class of planar graphs. We will prove that $S$ is quasi-planar by induction on $n$.
\comment{For convenience, let $\mathcal{CIG}^0(\cp):= \cp$, so that}
The base case $n=0$ is trivial. Now suppose $S$ is a \cig\ of a graph $S'\in \mathcal{CIG}^{n-1}(\cp)$. \comment{obtained by subdividing some edges of a graph $S''\in \mathcal{CIG}^{n-1}_*(\cp)$ into paths of bounded length.
Note that $S'$ and $S''$ are quasi-isometric.
}
By the induction hypothesis, $S'$ is quasi-isometric to some planar graph $P$.
\comment{and, since quasi-isometries are preserved under composition, so is $S'$.}
By applying \Cr{cor cigs} to $S'$ and $P$, we deduce that $S$ is quasi-isometric to some \cig\ of $P$, which is quasi-planar by \Tr{Davies}.
Since quasi-isometries are preserved under composition, we deduce that $S$ too is quasi-planar.
\end{proof}

\comment{
\begin{question}
    Suppose $G$ does not have $K_5$ or $K_{3,3}$ as an  asymptotic minor. Must \g be quasi-isometric to an iterated string graph?
\end{question}

This, combined with \Cr{cor isg}, would imply the coarse Kuratowski conjecture of \cite{GeoPapMin}.
}


We do not yet know if the converse of \Cr{cor isg} holds, and suspect it does not. However, we will be able to obtain a converse, and thus arrive at a fully combinatorial characterization of quasi-planar graphs, by allowing subdivisions of edges into bounded length paths  ---an operation that clearly preserves quasi-planarity--- each time we take a \cig. To make this precise, we define the class of iterated \defi{\scig s (\SCIG)} of $H$ by adapting the above definition of $\mathcal{CIG}^n$ as follows. Again we start with $\SCIG^0(H):=H$, and
we let $S^n(H)$ be the class of graphs obtained by subdividing each edge of a graph in $\SCIG^{n-1}(H)$ into a path of length at most 3,\footnote{Of length at most 2 is also possible by an adaptation of our construction.}
and let $\SCIG^n(H)$ denote the class of \cig s over a graph in $S^n(H)$.

We can now state our main result:

\begin{theorem} \label{pl iff subd iter str}
    A graph is quasi-isometric to a planar graph if and only if it is an iterated \scig\ of a planar graph.
\end{theorem}

We have sketched the proof of the backward direction above (the proof of \Cr{cor isg} easily extends to allow bounded-length subdivisions in each iteration). 

For the proof of the forward direction  we will introduce a notion of \defi{`edge-sliding'} which defines an equivalence relation on graphs of the same order, that we consider to be of independent interest; see \Sr{sec ES} for more. In fact this will yield a much more general result as we explain in the following subsection.


\subsection{Generalisation to, and beyond, minor-closed classes}

We have been focusing on quasi-planar graphs above, but our results are much more general. In particular, the forward direction of \Tr{pl iff subd iter str} follows from the forward direction of the following:

\begin{theorem} \label{qi implies iter scig}
    Let $G,H$ be graphs. The following are equivalent:
    \begin{enumerate}
        \item \label{qi-itercig i} $G,H$ are $(M,A)$-quasi-isometric; 
        \item  \label{qi-itercig ii} $G \in \SCIG^n (H)$ and $H \in \SCIG^n (G)$ for some $n \geq 0$, {whereby} at each iteration the cover consists of subgraphs with diameter bounded by some uniform $D \geq 0$ in their intrinsic metric.
    \end{enumerate}
    If $G,H$ satisfy the conditions above, then $M, A$ depend only on $n,D$ and vice-versa.
\end{theorem}

This is similar to \Tr{pl iff subd iter str}, but the bounded diameter condition is not present in the latter, and its role in the proof is replaced by \Tr{Davies}. James Davies (private communication) has expressed optimism that his \Tr{Davies} holds with planar graphs replaced by any minor-closed family. Once this is proved our \Tr{pl iff subd iter str} would extend immediately for any class $\cx = Forb(M)$, where $M$ is a collection of graphs; indeed, the forward implication via \Tr{qi implies iter scig} is unaffected, and the proof of the backward implication via \Cr{cor isg} would generalise.

Although we do not yet know whether subdivisions are necessary in \Tr{pl iff subd iter str}, we know they are in \Tr{qi implies iter scig}: when $H$ is a tree, any graph in $\CIG^n(H)$ has its cycle space generated by triangles, but this is not true for every quasi-tree $G$;  more details will be provided in subsequent work.

\subsection{Coarsely contraction closed classes}

As remarked above, \Tr{Davies} is `equivalent' to \Tr{cm to qp}, and we now generalise some of our results to graph classes having this property. 
Letting \defi{$\QI(\cx)$} denote the class of graphs that are quasi-isometric to some graph of a given class \cx, we have

\begin{corollary} \label{equivalence}
    The following are equivalent for any class of graphs $\cx$:
    \begin{enumerate}
    \item \label{ccc i} $\mathcal{QI}(\cx)$ is closed under contraction minors;
    \item \label{ccc ii} $\mathcal{CIG}(\cx) := \mathcal{CIG}^1(\cx) \subseteq \mathcal{QI}(\cx)$;
    \item \label{ccc iii} $\mathcal{CIG}^n(\cx) \subseteq \mathcal{QI}(\cx)$ for every $n \in \N$, and
    \item \label{ccc iv} $\SCIG^n(\cx) \subseteq \mathcal{QI}(\cx)$ for every $n \in \N.$
    \end{enumerate}
\end{corollary}

\begin{proof}
    Note that \ref{ccc i} and  \ref{ccc ii} are equivalent, where the forward direction follows from \Tr{cig to cm} and the backward direction follows from \Tr{cm to qp gen}.
    By proceeding as in \Cr{cor isg}, we can upgrade  
    \ref{ccc ii} to \ref{ccc iv}, and the latter clearly implies \ref{ccc iii} and \ref{ccc ii}.
\end{proof}

We call a class of graphs \cx\  \defi{\cococ} if it satisfies any (and hence all) of the conditions above.

For example, the class of planar graphs is coarsely contraction closed, since string graphs are quasi-planar, and above we expressed hope that so is any minor-closed family of finite graphs. Moreover, it is not hard to see that any class of graphs for which the (false \cite{DHIM}) fat minor conjecture of \cite{GeoPapMin} holds is coarsely contraction closed. We would be interested to see other sufficient (or necessary) conditions for a class to be \cococ.

In particular, letting \defi{$Forb_c(M)$} denote the class of graphs that forbid each member of a set $M$ of graphs as a contraction minor, one can ask 
\begin{question}
    Is $Forb_c(M)$ coarsely contraction closed for any non-empty set $M$ of finite graphs?
\end{question}

For example, for every $n \in \N$, the class of graphs with $n$ connected components is not minor-closed but it is closed under contraction minors.
It is easy to see that it is coarsely contraction closed.

\subsection{Tree-decompositions preserving quasi-planarity, and generalisations}

We will prove, in \Sr{sec TD QP}, that quasi-planarity is preserved under tree-decompositions with bounded-diameter adhesions: 
\begin{theorem} \label{tree decomposition}
    Suppose a graph $G$ admits a tree decomposition $(T,\{V_i\}_{i \in V(T)})$ \st\ for some {$M\in \R_{\geq 1}$ and $A \in \R_{\geq 0}$,}
    \begin{enumerate}
    \item \label{ptd i} each $G[V_i]$ is connected and $(M,A)$-quasi-planar \wrt\ $d_G$\footnote{The statement becomes false if we replace $d_G$ by the intrinsic metric $d_{G[V_i]}$ here; however, given such a tree decomposition, we can always modify it so that $d_{G[V_i]}$ and $d_G$ differ by bounded additive distortion, in which case it does not matter which of the two metrics we consider. The same applies to item \ref{ptd ii}; see \Lr{lem:TD exp} and \Rr{rem TD}. Alternatively, if in \ref{ptd ii} we require bounded intrinsic rather than extrinsic diameter, then in \ref{ptd i} it does not matter whether we consider the intrinsic or the extrinsic metric on each {induced} bag. }, and
    \item \label{ptd ii} the adhesion sets $G[V_i] \cap G[V_j]$ have uniformly bounded diameters in $d_G$.
    \end{enumerate}
    Then $G$ is quasi-planar.
\end{theorem}
Interestingly, the proof of this relies on \Tr{Davies}, and uses \cig s similarly to \Tr{cm to qp gen}.

More generally, we say that a class of graphs \cx\ is \defi{tree-closed} if a graph belongs to \cx\ if and only if it admits a tree decomposition of adhesion $1$ where each induced bag belongs to \cx. For example, any $\cx = Forb(M)$, where $M$ is a collection of 2-connected graphs is tree-closed.

Then \Tr{tree decomposition} can be stated more generally by replacing the class of planar graphs with any class which is coarsely contraction closed and tree-closed; see \Tr{tree decomposition gen}.

\subsection{Applications to groups}

MacManus \cite[Corollary~D]{McMAcc} proved that a finitely generated group is quasi-planar (i.e.\ it admits a finitely generated quasi-planar \Cg) \iff\ it is virtually-planar (i.e.\ it has a finite-index subgroup which admits a finitely generated planar \Cg). Combining this with \Tr{tree decomposition}, we will easily obtain 
the following:

\begin{corollary}
    Let $\Gamma$ be a finitely generated group that splits over a finite subgroup into virtually-planar factors\footnote{This means that $\Gamma$ is an  amalgamated free product $H *_C H'$ or an HNN extension $H *_\alpha$ whereby $H,H'$ are virtually-planar. We assume familiarity with these concepts; see e.g.\ \cite{LyndonSchupp} for an elementary treatment.}. Then $\Gamma$ is virtually-planar too. 
\end{corollary}
\begin{proof}[Proof (sketch)]
Suppose $\Gamma$ is an  amalgamated free product $H *_C H'$ with $C$ finite (the HNN case is similar). Let $S,S'$ be finite generating sets of $H,H'$ respectively, each containing all of $C$, and let $G$ be the \Cg\ of $\Gamma$ \wrt\ $S \cup S'$. Easily, \g admits a tree decomposition where each bag is a coset of either $H$ or $H'$, and each adhesion set is a coset of $C$. Applying \Tr{tree decomposition} to this tree decomposition we deduce that $G$ is quasi-planar, hence virtually-planar by MacManus's aforementioned theorem \cite{McMAcc}.
\end{proof}
    
This is probably known to experts, and Joseph MacManus (private communications) informs us that it can also be proved using Bass-Serre theory and known facts about (virtual) surface groups, but the coarse-graph-theoretic proof we provided may appeal more to some readers. 

\medskip

One of the most fundamental facts of geometric group theory is that if $\Gamma'$ is a finite index subgroup of a group $\Gamma$, then any two finitely generated \Cg s of $\Gamma, \Gamma'$ are \qic. Using our edge-sliding technique similarly to the proof of \Tr{pl iff subd iter str} we will obtain the following strengthening: 

\begin{corollary} \label{cayley}
    Let $\Gamma$ be a finitely generated group and let $\Gamma' \leq \Gamma$ be a finite index subgroup. Let $S, S'$ be generating set of $\Gamma, \Gamma'$ respectively. Then the Cayley graph $G = Cay(\Gamma, S)$ is an iterated \scig\ of $G'= Cay(\Gamma',S')$, using at each step a cover that is $\Gamma'$-invariant.
\end{corollary}

\section{Preliminaries}

We follow the terminology of Diestel~\cite{DiestelBook25} for graph-theoretic terms, and that of  Godsil \& Royle \cite{GR01} for the few group theoretic ones that we will need.
Throughout the paper, we assume all graphs to be countable.

\subsection{Graphs} \label{sec gra}

We say that $G$ is a \defi{contraction minor} of a graph $H$, if \g is obtained from $H$ by contracting edges. If the maximal connected subgraphs that are contracted have diameter bounded by some $r \in \R$, we say that $G$ is an \defi{$r$-shallow contraction minor} of $H$.

\medskip

Given a family of sets $\{S_i\}_{i\in \ci}$, we define its \defi{\ing} to be the graph $S$ with vertex set $\ci$ in which $S_i S_j$ is an edge whenever $S_i$ and $S_j$ intersect. 




\subsection{Distances} \label{sec dist}

A \defi{geodesic metric space} is a metric space $(X,d)$ such that for every $x,y \in X$ there is an $x\text{-}y$ arc of length $d(x,y)$. 

Let $G$ be a graph.
We write~\defi{$d_G(v, u)$} for the distance of the two vertices~$v$ and~$u$ in~$G$. Note that this turns \g into a geodesic metric space. In particular, for any $x, y \in V(G)$ such that $d(x,y)< \infty$, there is a \defi{geodesic path} $P$ of length ${||P||} = d(x,y)$.

Given a graph \g and $K \in \N$, we denote with $G^K$ the $K$-th power of $G$, i.e.\ the graph on $V(G)$ with edge-set $\{xy \mid d_G(x,y)\leq K\}$.

For two sets~$U$ and~$U'$ of vertices of~$G$, we write~\defi{$d_G(U, U')$} for the minimum distance of two elements of~$U$ and~$U'$, respectively.
If one of~$U$ or~$U'$ is just a singleton, then we omit the braces, writing $d_G(v, U') := d_G(\{v\}, U')$ for $v \in V(G)$.


Given a set~$U$ of vertices of~$G$, the \defi{ball (in~$G$) around~$U$ of radius $r \in \N$}, denoted by~\defi{$B_G(U, r)$}, is the set of all vertices in~$G$ of distance at most~$r$ from~$U$ in~$G$.
If~$U = \{v\}$ for some~$v \in V(G)$, then we omit the braces, writing~$B_G(v, r)$ for the ball (in $G$) around~$v$ of radius~$r$.


\medskip

For $M \in \R_{\geq 1}$ and $A \in \R_{\geq 0}$, an \defi{$(M, A)$-quasi-isometry} from a graph~$G$ to a graph~$H$ is a map~$\varphi : V(G) \rightarrow V(H)$ such that
\begin{enumerate}[label=\rm{(Q\arabic*)}]
    \item \label{quasiisom:1} $M^{-1} \cdot d_G(g,h) - A \leq d_H(\varphi(g),\varphi(h)) \leq M\cdot d_G(g,h)+A$ for every $g,h \in V(G)$, and
    \item \label{quasiisom:2} for every $v \in V(H)$ there is $h \in V(G)$ such that $d_H(v,\varphi(h)) \leq A$.
\end{enumerate}

When this holds for $M=1$, we say that $\varphi$ has \defi{additive distortion} (at most) $A$.

\medskip

Let $(X,d)$ be a metric space and let $\{x_i\}_{i \in \ci}$ be a collection of points of $X$. For each $i \in \ci$, the \defi{Voronoi cell} $R_i$ is the set of all points in $X$ whose distance from $x_i$ is not greater than their distance from any other $x_j$.
Namely, \\ $R_i=\left\{x \in X \mid d\left(x, x_i\right) \leq d\left(x, x_j\right) \text { for all } j \neq i\right\}$.

Note that Voronoi cells cover $X$ and, after an appropriate refinement to avoid overlapping, we can assume they give a partition of $X$. We will call such a partition a \defi{Voronoi cell decomposition} of $X$ with respect to $\{x_i\}_{i \in \ci}$. Easily, if $(X,d)$ is a geodesic metric space, then $R_i$ is always connected, as it contains a geodesic from each $x\in R_i$ to $x_i$.

\medskip


\subsection{Tree-decompositions} \label{sec TD}

A \defi{tree-decomposition} of a graph $G$ is a pair $(T, (V_t : t \in V(T)))$, where $T$ is a tree, and $V_t$ is a subset of $V(G)$ for each $t \in V(T)$, such that:
\begin{enumerate}
    \item \label{TDi} $V(G) = \bigcup_{t \in V(T)} V_t$;
    \item \label{TDii}  for every edge $e = uv$ of $G$, there exists $t \in V(T)$ with $u, v \in V_t$; and
    \item \label{TDiii} for all $t_1, t_2, t_3 \in V(T)$, if $t_2$ lies on the path of $T$ between $t_1, t_3$, then $V_{t_1} \cap V_{t_3} \subseteq V_{t_2}$.
\end{enumerate}
We will refer to each set $V_t$ as a \defi{bag} {and to each induced subgraph $G[V_t]$ of $G$ as an  \defi{induced bag}}.

\section{Proofs of \Trs{cm to qp gen} and \ref{cig to cm}}

We will now prove our results showing that contraction minors and \cig s differ by quasi-isometries.
For \Tr{cm to qp gen}, we start with a simple lemma. The \defi{Hausdorff distance $\Hd( X , Y )= \Hd_G( X , Y )$} between 
subsets $X,Y$ of a metric space $G$ is defined as $\max \left\{ \sup_{x \in X} d(x,Y) ,   \sup_{y\in Y} d(X , y) \right \}$. 


\begin{lemma} \label{lem confn}
    Let $f: V(G) \rightarrow V(H)$ be a $(M,A)$-quasi-isometry between graphs \g and $H$. Let $C$ be a connected subgraph of $G$, and let $C'=B_H(f(C), M+A)$. 
    Then $C'$ is connected and    
    %
    %
    %
    \labequ{confn i}{\Hd_G(C,f^{-1}(B_H(C',A)))\leq M^2+3AM.}
\end{lemma}

\begin{proof}
    To show that $C'$ is connected, it suffices to check that for any $x,y \in C$, there is a $f(x)$-$f(y)$ path contained in $C'$.
    This follows easily from the fact that $f$ is a $(M,A)$-quasi-isometry, and any adjacent vertices on a $x$-$y$-path in $G$ are mapped by $f$ to vertices at distance at most $M+A$ in $H$.

    To prove \eqref{confn i}, pick an arbitrary $x \in f^{-1}(B_H(C',A))$.
    Easily $\Hd_H(f(C),B_H(C',A))\leq$ \\$M+2A$, hence we can find some $x' \in C$ such that $d_H(f(x'), f(x)) \leq M+2A$.
    Since $f$ is a quasi-isometry, we also have $d_G(x',x) \leq M d_H(f(x'), f(x)) +AM \leq M(M+2A) + AM$.
    Hence $d_G(C, x) \leq M^2+3AM$.
\end{proof}

Now, we prove \Tr{cm to qp gen}:

\begin{proof}[{Proof} of \Tr{cm to qp gen}]
Let $f: G \rightarrow H$ be a $(M,A)$-quasi-isometry. Let $G'$ be a contraction minor of \G, and let $\cc= \{C_i, i \in \ci\}$ be the set of components of the subgraph of \g spanned by the contracted edges. 
Thus $G'=G/\cc$, i.e.\ $G'$ is obtained from $G$ by contracting each element of \cc\ into a vertex.
We may assume \obda\ that $\bigcup_{i \in \ci} C_i$ contains all vertices of \g by adding singletons to \cc\ until this is satisfied.

Let $C_i' := B_H(f(C_i),M+A)$ for ach $i\in \ci$.
Let $S$ be the cover-intersection-graph of $H$ with respect to the collection $\cs:= \{C'_i \mid i\in \ci\} \cup E(H)$ of subgraphs of $H$. Note that $S$ is a cover intersection graph of $H$ by definition. We will show that $g: V(G') \to V(S)$ defined by $C_i \mapsto C'_i$ is a quasi-isometry.

To begin with, note that \fe\  $i,j \in \ci$, we have
    \begin{equation} \label{dSdHi}
        d_S(C_i', C_j') \leq  d_H(C_i', C_j') + 1        
    \end{equation}
because if $Q$ is a \pth{C_i'}{C_j'}\ in $H$, then its edges are vertices of $S$ that combined with $C_i',C_j'\in V(S)$ form a \pth{C_i'}{C_j'}\ in $S$ (this is the reason why we added $E(H)$ to $V(S)$).
    
To check that $g$ satisfies \ref{quasiisom:1}, pick $C_k,C_m\in V(G')$, and a shortest \pth{C_k}{C_m}\ $P$ in $G'$, so that $d_{G'}(C_k,C_m)=||P||$.
Note that for every pair of consecutive vertices $C_i,C_j$ of $P$ we have $d_H(C'_i,C'_j) \leq d_H(f(C_i),f(C_j))\leq M+A$, and combining this with \eqref{dSdHi} we deduce $d_S(C'_i,C'_j)\leq M+A+1$. Applying this to each edge of $P$ we obtain
$$d_{S}(C'_k,C'_m)\leq (M+A+1)\cdot||P||= (M+A+1) \cdot d_{G'}(C_k,C_m).$$
Conversely, pick a shortest \pth{C'_k}{C'_m}\ $P'=(C'_k=)v_0 v_1 \ldots v_n (=C'_m)$ in $S$, so that $d_{S}(C'_k,C'_m)=||P'||$. Pick a vertex $x_i\in v_{i-1} \cap v_i$ of $H$ \fe\ $1\leq i \leq n$, which exists since $v_{i-1} v_i\in E(S)$. Let $\kreis{x}_i$ be a vertex of $G'$ corresponding to some vertex $x'_i$ (of \G) in $f^{-1}(B_H(x_i,A))$, which exists since $f$ is $A$-quasi-surjective. Note that \fe\ $i$, both $x_{i-1} x_i$ belong to an edge of $H$ or to some $C'_r$.
If the former is the case, we have $d_{G'}(\kreis{x}_{i-1}, \kreis{x}_i)\leq d_G(x_{i-1}', x_{i}') \leq M(d_H(f(x_{i-1}'),f(x_i'))+A) \leq M(d_H(x_{i-1}, x_i)+3A) \leq M(3A+1)$.
If the latter is the case, we have $d_{G'}(\kreis{x}_{i-1}, \kreis{x}_i) \comment{\leq d_{G}(x'_{i-1}, x'_i)} \leq d_{G}(x'_{i-1}, C_r) + d_{G}(x'_{i}, C_r)$ $ \leq 2 M^2 + 6AM$ by \eqref{confn i} from \Lr{lem confn}.
Moreover, both $d_{G'}(C_k,\kreis{x}_0), d_{G'}(C_m,\kreis{x}_n)$ are also upper-bounded by $M^2 + 3AM$ for the same reason. Applying one of these inequalities to each edge of $P'$ we deduce 
$$d_{G'}(C_k,C_m) \leq (2 M^2 + 6AM) \cdot ||P'|| = (2 M^2 + 6AM) \cdot d_{S}(C'_k,C'_m).$$ 

\medskip
To check the quasi-surjectivity property \ref{quasiisom:2}, for any edge $e \in V(S)$ consider one of its endpoints $y \in V(H)$. Since $f$ is $A$-quasi-surjective, there is some $x \in V(G)$ such that $y \in B_H(f(x),A)$. Choosing the $i\in \ci$ such that $x \in V(C_i)$, we have $d_H(e, C_i') \leq A$, and so $d_S(e, C'_i) \leq A+1$.

Thus $g$ is an $(M',A')$-quasi-isometry for $A':=A+1$ and $M':=2M^2+6AM+1$, where we used the fact that $M\geq 1$ to deduce that $M'\geq M+A+1$.
\end{proof}

One can adapt the proof of \Tr{cm to qp gen} to the more general setting where $G$ a geodesic metric space (not necessarily a graph), and, by applying \Tr{Davies}, obtain the following strengthening {of \Tr{cm to qp}}.

\begin{corollary} \label{cor: metr contr}
    Let $G$ be a geodesic metric space quasi-isometric to a planar graph. Then every space $G'$ obtained from $G$ by contracting connected sub-spaces is quasi-isometric to a planar graph.
\end{corollary}

\begin{proof}
    We will follow the lines of \Tr{cm to qp gen}, except that $H$ is now planar, and we will apply \Tr{Davies} to the resulting \cig\ of $H$.
    
    We start with an easy remark that will allow us to adapt any argument involving ``adjacent vertices'' on a geodesic to our more general setting:
    if $G$ is a geodesic metric space and $x,y \in G$ with $d_G(x,y) < \infty$, then
    \labtequ{rem geod metr}{
        there are $x=x_0, \dots, x_{\lceil d_G (x,y) \rceil}=y$ in $G$ such that $d_G(x_i, x_{i+1}) \leq 1$ for every $1 \leq i \leq \lceil d_G (x,y) \rceil$.
    }
    \Lr{lem confn} can be easily generalized to the case where $G$ is a geodesic metric space and $C$ is a connected sub-space, by applying \eqref{rem geod metr} when checking that $C'$ is connected.

    As for the proof of \Tr{cm to qp gen}, note that, to apply \Tr{Davies}, the set $\cc$ has to be countable.
    For this reason, we cannot require $\bigcup \cc$ to be the whole $G$ (which in general can be uncountable).
    However, since $G$ is quasi-isometric to a countable graph $H$, we can ensure that every point in $G$ is at distance at most some $R\in \R_{\geq 0}$ (depending on $M,A$ only) from $\bigcup \cc$, by adding singletons to $\cc$ until this is satisfied.
    As a result, the definition of $g: G' \to V(S)$  has to be extended as follows: for every point $x' \in G'$, corresponding to some $x \in G$, pick a $C_i$  such that $d_G(x, C_i)\leq R$, 
    and let $g(x') = C_i'$.

    To check that $g$ satisfies \ref{quasiisom:1}, the proof proceeds similarly to \Tr{cm to qp gen}.
    The first bound follows by adapting the existing argument with \eqref{rem geod metr}. 
    We deduce that, for any $x,y \in G'$, we have
    $$d_{S}(g(x),g(y)) \leq (M(2R+1)+A+1) \cdot \lceil d_{G'} (x,y) \rceil \leq (M(2R+1)+A+1) \cdot (d_{G'}(x,y)+1).$$
    The second bound for all points in $\cc$ follows with the same proof as in \Tr{cm to qp gen}.
    Hence for every $x,y, \in G'$, we have 
    $$d_{G'}(x,y) \leq (2 M^2 + 6AM) \cdot d_{S}(g(x), g(y))+2R.$$ 

    For the quasi-surjectivity property \ref{quasiisom:2}, a similar argument implies the existence of some $C_i$ such that
    $d_S(e, C'_i) \leq MR + 2A+1$.
\end{proof}



Next, we prove our approximate converse of  \Tr{cm to qp gen}:

\begin{proof}[Proof of \Tr{cig to cm}]
    Let $\{S_i\}_{i\in \ci}$ be a cover  witnessing that $S$ is a \cig\ of \G. Construct a graph $G'$ from \g as follows. For each $v\in V(G)$, we replace $v$ by a clique $K^v$ on $n_v$ vertices, where \defi{$n_v$} is the \defi{ply} of $v$ in $\{S_i\}$, i.e.\ the number of $S_i$ containing $v$. We call the edges within each $K^v$ the \defi{vertical edges} of $G'$. We label the vertices of $K^v$ as $v_{S_i}, i\in \ci_v$, where $\ci_v$ is the set of indices $i\in \ci$ with $v\in S_i$. Finally, for each $i\in \ci$, and any edge $vw\in E(S_i)$, we add a $v_{S_i}$-$w_{S_i}$~edge to $G'$, which we call a \defi{horizontal edge}. Note that since each $vw\in E(G)$ lies in some $S_i$, we have
    \labtequ{KvKw}{$d_{G'}(K^v,K^w)=1$ \iff\ $vw\in E(G)$.}

    This completes the construction of $G'$. We claim that it has the desired properties. Indeed, let $\pi: V(G') \to V(G)$ be the canonical projection, i.e.\ $\pi$ maps each $K^v$ to $v$. 
    Easily, we have $d_{G'}(x,y) \geq d_G(\pi(x),\pi(y))$ \fe\ $x,y\in V(G')$, and by \eqref{KvKw} we have  
    $d_{G'}(x,y) \leq 2d_G(\pi(x),\pi(y))+1$; indeed each \pth{\pi(x)}{\pi(y)}\ $P$ in $G$ can be transformed into an \pth{x}{y}\ in $G'$  using at most one edge within each $K^v$ visited by $P$. Thus $\pi$ is a $(2,1)$-quasi-isometry. 

\medskip
    It remains to check that $S$ is isomorphic to a contraction minor of $G'$. To see this, define \defi{$S'_i$} \fe\ $i\in \ci$ by replacing each vertex $v\in V(S_i)$ by $v_{S_i}$, and each edge $vw\in E(S_i)$ by $v_{S_i} w_{S_i}$. Note that $S'_i$ is connected since $S_i$ is.
    We claim that the graph $S':= G'/\{S'_i\}$ obtained by contracting each $S'_i, i\in \ci$ to a point is isomorphic to $S$. Indeed, the $S'_i$ are pairwise disjoint by construction, and they cover $V(G')$. Thus the map $i\mapsto S'_i$ is a bijection from $V(S)$ to $V(S')$. To see that this bijection yields an isomorphism from $S$ to $S'$, note that each horizontal edge of $G'$ lies in exactly one $S'_i$, and there is a (vertical) $S'_i$-$S'_j$ edge \iff\ $S'_i,S'_j$ intersect a common clique $S^v$, which happens \iff\ $S_i\cap S_j \neq \emptyset$ and therefore  $ij\in E(S)$.
\end{proof}

In case the sets of the cover in \Tr{cig to cm} have uniformly bounded diameters, we easily obtain the following strengthening:

\begin{corollary} \label{bounddiamsets}
    Let $S$ be a \cig\ of a graph $G$, {whereby} the cover of $G$ consists of sets with diameter at most $D$. Then $S$ is $(2D+2, 1)$-quasi-isometric to $G$.
\end{corollary}

\begin{proof}
    We follow the lines of the proof of \Tr{cig to cm}, noting that if the sets $S_i$ have diameter at most $D$, then so do the $S_i'$. Easily, the resulting contraction minor $S'$ is $(D+1,0)$-quasi-isometric to $G'$. Since $G'$ is $(2,1)$-quasi-isometric to $G$, we deduce that $S=S'$ is $(2D+2, 1)$-quasi-isometric to $G$.
\end{proof}

\section{Tree decompositions preserving quasi-isometries} \label{sec TD QP}

In this section we prove \Tr{tree decomposition}.
We start with a simple lemma.


\begin{lemma} \label{lem:TD path}
    Let $(T, (V_t : t \in V(T)))$ be a tree-decomposition of a graph~$G$. Let $k > 1$, $t \in V(T)$, and $P=(x_0 \dots x_k)$ be a path in $G$, with endpoints $x_0,x_k \in V_t$ and $x_i \not \in V_t$ for every $1 \leq i \leq k-1$. Then $x_0, x_k$ belong to some adhesion set $V_{t} \cap V_{t'}$, where $tt' \in E(T)$.
\end{lemma}

\begin{proof}
By standard properties of tree-decompositions (see e.g.\ \cite[Theorem 10.13]{AlgDes}),  for any node $t \in T$, the graph $G - V_t$ consists of pairwise disjoint and non-adjacent subgraphs, each one corresponding to a connected component of $T - t$.
In particular, there is exactly one connected component of $T - t$, say $T'$, such that $G[ \bigcup_{i \in V(T')}V_{i}] - V_t$ intersects the connected set $P - V_t$.
Since $x_0x_1$ and $x_{k-1}x_k$ are edges, there exist $t_0, t_k \in V(T)$ such that $x_0,x_1 \in V_{t_0}$ and $x_{k-1},x_k \in V_{t_k}$.
As $x_1, x_{k-1} \in P - V_t$, we deduce that $t_0, t_k \in V(T')$.
Let $t'$ be the unique neighbour of $t$ in $T'$, and note that the $t$-$t_0$ path in $T$ passes through $t'$.
Since $x_0 \in V_t \cap V_{t_0}$, by the coherence property of tree-decompositions, we have that $x_0 \in V_t \cap V_{t'}$, which is an adhesion set.
Proceeding in the same way for $x_k$, our claim follows.
\end{proof}

The following lemma improves the metric properties of our tree-decompositions by enlarging their bags.
Given $X\subset V(G)$ and $r\in \R_+$, let ${X^{+r}}:= B_G(X,r)$.

\begin{lemma} \label{lem:TD exp}
    Let $(T, (V_t : t \in V(T)))$ be a tree-decomposition of a graph~$G$, and $r\in \R_+$. Then $(T, (V_t^{+r} : t \in V(T)))$ is also a tree-decomposition of~$G$, and its adhesions satisfy 
    \begin{equation} \label{eq: ad sets}
        (V_i \cap V_j)^{+r} = V_i^{+r} \cap V_j^{+r}.
    \end{equation}
    Moreover, if for each adhesion set $A = V_i \cap V_j$ we have $diam_G(A) \leq r$, then
    \begin{enumerate}
        \item \label{ltd i} for every adhesion set $A'=V_i^{+r} \cap V_j ^{+r}$ we have $diam_{G[V_i^{+r}]}(A')\leq 3r$; 
        \item \label{ltd ii} the identity $\Id:G[V_t^{+r}] \to G$ has additive distortion at most $r':=4r$ \fe\ $t\in V(T)$.\footnote{A very similar lemma appears in \cite{AlbGeoBlo}.}
    \end{enumerate}
\end{lemma}

\begin{proof}
    It is proved in \cite[Lemma 3.3]{StrDualPath} that $(T, (V_t^{+r} : t \in V(T)))$ is a tree-decomposition of~$G$.
    To prove \eqref{eq: ad sets}, the inclusion $(V_i \cap V_j)^{+r} \subseteq V_i^{+r} \cap V_j^{+r}$ obviously holds for any vertex-sets $V_i,V_j$ of a graph. 
    For the converse we use the properties of a tree decomposition as follows.
    Pick $x \in V_i^{+r} \cap V_j^{+r}$, and paths $P_{x,V_i}, P_{x,V_j}$ from $x$ to some elements $x_i \in V_i$ and $x_j \in V_j$ of length at most $r$.
    Any $x_i$-$x_j$-path contained in $P_{x,V_i} \cup P_{x,V_j}$ will intersect $V_i \cap V_j$, as the latter separates $V_i$ from $V_j$ \cite[Lemma 12.3.1]{DiestelBook25}. 
    Let $y$ be a point in this intersection. Then $d_G(x,y) \leq r$ and so $x \in (V_i \cap V_j)^{+r}$.

\smallskip
    To prove \ref{ltd i},
    consider $x,y \in V_i^{+r} \cap V_j^{+r}$.
    By \eqref{eq: ad sets}, we have $x,y \in (V_i \cap V_j)^{+r}$. Thus there are two paths from $x,y$ to some $x',y' \in V_i \cap V_j$ of length at most $r$.
    But $d_G(x',y') \leq r$ by our assumption, so $x,y$ are connected by a path in $(V_i \cap V_j)^{+r}$ of length at most $r$.
    Thus there is an $x$-$y$-path contained in $(V_i \cap V_j)^{+r}$, hence, by \eqref{eq: ad sets}, in both $ V_i^{+r},  V_j^{+r}$, of length at most $3r$.

\smallskip
    For \ref{ltd ii}, clearly $d_G \leq d_{G[V_t^{+r}]}$.
    On the other side, consider two points $x,y \in V_t^{+r}$ and let $x',y' \in V_t$ such that $d_G(x,x') = d_{G[V_t^{+r}]} (x,x') \leq r$ and $d_G(y,y') = d_{G[V_t^{+r}]} (y,y') \leq r$.
    Let $P$ be a geodesic $x'$-$y'$-path in $G$.
    We claim that $P$ is contained in $V_t^{+r}$, from which we can deduce that $d_G(x',y') = d_{G[V_t^{+r}]}(x',y')$ and conclude that
    $$d_{G[V_t^{+r}]}(x,y) \leq d_{G[V_t^{+r}]}(x',y') + 2r = d_{G}(x',y') + 2r \leq d_G(x,y) + 4r.$$
    To prove the claim, partition $P$ in sub-paths $P_1, \dots, P_K$ where each $P_k$ has endpoints $x_k, y_k$ in $V_t$ and inner vertices (if existing) not in $V_t$. If such inner vertices do not exist, then $P_k$ is contained in $V_t$. Otherwise, by \Lr{lem:TD path}, the endpoints $x_k, y_k$ must belong to some adhesion set, hence $d_G(x_k, y_k) \leq r$ and so $P_k$ is contained in $V_t^{+r}$.
\end{proof}

We now prove \Tr{tree decomposition} in the following more general form as described in the introduction:

\begin{theorem} \label{tree decomposition gen}
    Let \cx\ be a class of graphs which is tree-closed and coarsely contraction closed. Let $G$ be a graph admitting a tree decomposition $(T,\{V_i\}_{i \in V(T)})$ \st\ for some {$M \in \R_{\geq 1}$ and $A,r \in \R_{\geq 0}$},
    \begin{enumerate}
    \item each $G[V_i]$ is connected and $(M,A)$-quasi-isometric to a graph in \cx\ \wrt\ $d_G$, and
    \item each adhesion set $G[V_i] \cap G[V_j]$ has  diameter at most $r$ in $d_G$.
    \end{enumerate}
    Then $G$ is $(M',A')$-quasi-isometric to a graph in \cx, with $M',A'$ depending on $M,A,r$ and \cx.
\end{theorem}

\begin{proof}

    We may assume \obda\ that \g is connected, and each bag $V_i$ is non-empty. 
    By \Lr{lem:TD exp}, we can further assume that the adhesion sets $V_i \cap V_j$ have bounded diameter in the intrinsic metrics of $G[V_i]$ and $G[V_j]$ and each induced bag $G[V_i]$ is $(M,A)$-quasi-isometric to a graph in \cx\ \wrt\ its intrinsic metric.

    For each $i \in V(T)$, let $f_i: V_i \rightarrow C_i$ be a $(M,A)$-quasi-isometry where $C_i \in \cx$. We say that $f_i$ is a \defi{local quasi-isometry}.
    We glue the graphs $\{C_i\}_{i \in V(T)}$ together as follows: for every edge $ij \in E(T)$, choose an arbitrary $v_{ij}=v_{ji}$ in the non-empty set $V_i \cap V_j$ and identify the \defi{junction points} $f_i(v_{ij}) \in C_i$ and $f_j(v_{ij}) \in C_j$. Let $C$ be the resulting graph, which lies in \cx\ since the latter is tree-closed.

    We aim to construct a global quasi-isometry from \g to a  graph in \cx. A first attempt could be to map $G$ to $C$ by sending each vertex to one of its images under the local quasi-isometries. However, this approach may fail when vertices have images in $C$ that are arbitrarily far apart.  
    Therefore, we instead construct a quasi-isometry from \g to a \cig\ of $C$, where, for each vertex of $G$, the cover groups together all its images under the local quasi-isometries, and follow the lines of the proof of \Tr{cm to qp gen}.
    
    For every $v \in V(G)$, let $T_v$ denote the connected sub-tree of $T$ whose vertices index the bags containing $v$. We construct a tree $T_v'$ in $C$ by joining with geodesic paths the vertices in
    $\{f_i(v)\}_{i \in V(T_v)} \cup \{f_i(v_{ij})=f_j(v_{ij})\}_{ij \in E(T_v)}$.
    More precisely, every $T_v'$ is the union of edge-disjoint trees $T_{v,i}'$ contained in $C_i$, for each bag $i \in V(T_v)$, constructed as follows. Each $T_{v,i}'$ is a nested union of trees $\{T_{v,i,j}'\}_{j \geq 0}$ defined inductively as follows. 
    We start from the one-vertex tree $T_{v,i,0}' = f_i(v)$ and fix an enumeration $\{y_j\}_{j \geq 1}$ of the neighbours of $i$ in $T_v$. At step $j \geq 1$, we construct $T_{v,i,j}$ by picking a geodesic $Q$ from the junction point $f_i(v_{iy_j})$ to $f_i(v)$, and extending $T_{v,i,j-1}$ by adding the subpath of $Q$ from $f_i(v_{iy_j})$ up to its first vertex in $T_{v,i,j-1}$.

    We will need the following consequence of this construction: for every $x\in V(T_v')$, there is a $C_j$ such that $x \in C_j$ and
    \begin{equation} \label{trees}
        d_{C_{j}}(f_{j}(v), x) \leq Mr+A.      
    \end{equation}
    Indeed, if $T_v$ has no edges, then $x=f_j(v)$, where $j$ is the only vertex of $T_v$. Otherwise, $x$ lies on a geodesic between $f_j(v)$ and $f_j(v_{ij})$ for some $ij \in E(T_v)$. Thus $v, v_{ij} \in V(G)$ belong to $V_i \cap V_j$ and so $d_{G[V_j]}(v,v_{ij}) \leq r$. Using the local quasi-isometry $f_j$, we deduce $d_{C_j}(f_j(v),f_j(v_{ij})) \leq Mr + A$ and, since $x$ lies on a $f_j(v)$--$f_j(v_{ij})$~geodesic, we obtain $d_{C_j}(f_i(v),x) \leq Mr + A$.

\medskip
    Let $C'$ denote the \cig\ of $C$ with respect to the cover $\cc = \{T_v'\}_{v \in V(G)} \cup E(C)$. Then $C'\in \CIG(\cx)$, and hence $C'\in \QI(\cx)$ since \cx\ is assumed to be coarsely contraction closed.
    We now claim that the map $f: G \rightarrow C'$ that maps any vertex $v$ to the corresponding $T_v'$ is a quasi-isometry, from which we deduce that $G\in \QI(\cx)$ as desired. The proof is similar to that of \Tr{cm to qp gen}.

    We first check the quasi-embedding property \ref{quasiisom:1}. On one side, let $x, y \in V(G)$ and consider an $x\text{-}y$~geodesic $P=(x=)x_0x_1...x_n(=y)$ in $G$. Recall that every edge $x_ix_{i+1}$ is contained in some $V_{j}$.
    Note that $d_{C'}(T_{x_i}', T_{x_{i+1}}') \leq d_{C}(T_{x_i}', T_{x_{i+1}}')+1$ because we included all the edges of $C$ in the cover \cc. Moreover, the right-hand-side is upper bounded by $d_{C}(f_{j}(x_i),f_{j}(x_{i+1}))+1$. Since $f_j$ is a (local) quasi-isometry, we can bound the latter by $Md_{G[V_{j}]}(x_i, x_{i+1})+A+1 = M+A+1$.
    We conclude that $d_{C'}(T_{x}', T_{y}') \leq (M+A+1) ||P|| = (M+A+1) d_G(x,y)$.

    On the other side, pick a geodesic $P'=(T_x'=)v_0v_1...v_n(=T_y')$ in $C'$. For every edge $v_iv_{i+1}$, pick some $x_i \in C$ belonging to the non-empty intersection $v_i \cap v_{i+1}$. We choose a suitable ``(quasi-)pre-image'' $\kreis{x}_i \in G$ for every $x_i$ as follows: if $x_i \in C$ is a junction point, let $Y = \{ j \in V(T) \mid x_i \in C_j\}$. By construction, there is some $\kreis{x}_i \in  \bigcap_{j \in Y} V_{j}$ such that $x_i = f_{j}(\kreis{x}_i)$ for every $j \in Y$. If $x_i$ is not a junction point, choose a $\kreis{x}_i$ such that its image under the local quasi-isometry is close to $x_i$. Thus, whenever $x_i$ belongs to some $C_{j}$, we have $\kreis{x}_i \in V_{j}$ and
    \begin{equation} \label{choosekreis}
        d_{C_j}(x_i, f_{j}(\kreis{x}_i)) \leq A.
    \end{equation}
    Note that for each $0\leq i\leq n-1$, the points $x_i, x_{i+1}$ belong to either an edge of $C$ or to some $T_v'$. If the former is the case, say $x_ix_{i+1} \in E(C_{j})$ and hence $\kreis{x_i}, \kreis{x_{i+1}} \in V_{j}$. We have 
    $d_G(\kreis{x_i}, \kreis{x_{i+1}}) \leq d_{G[V_{j}]}(\kreis{x_i}, \kreis{x_{i+1}}) \leq  M(d_{C_{j}}(f_{j}(\kreis{x_i}), f_{j}(\kreis{x_{i+1}})) + A)$. Applying \eqref{choosekreis} twice, we bound the latter by $M(d_{C_{j}}(x_i, x_{i+1})) + 3A) = 3AM + M$.
    If the latter is the case, apply separately \eqref{trees} to $x_i$ and $x_{i+1}$. For $x_i$ we thus find a $C_j$ such that $x_i \in C_j$ with $d_{C_j}(f_j(v), x_i) \leq Mr + A$.
    Combining this with \eqref{choosekreis}, we have $d_{C_j}(f_j(v), f_j(\kreis{x_i})) \leq Mr +2A$.
    Using the local quasi-isometry, we have $d_G(v, \kreis{x_i}) \leq d_{G[V_j]}(v, \kreis{x_i}) \leq M(Mr +2A +A) = M^2r + 3AM$. For the same reason, also $d_G(v, \kreis{x_{i+1}})$ is bounded by the same constant and combining those we have $d_G(\kreis{x_i}, \kreis{x_{i+1}}) \leq  2M^2r + 6AM$.
    Recalling that $x_0 \in T_x'$ and $x_n \in T_y'$, also $d_G(x, \kreis{x_0})$ and $d_G(y, \kreis{x_n})$ are bounded by $M^2r + 3AM$.
    Thus $d_G(x,y) \leq (2AM^2r + 6AM)||P'|| =  (2AM^2r + 6AM) d_{C'}(T_x', T_y')$.

    To check the quasi-surjectivity property \ref{quasiisom:2}, for any edge $e \in E(C) \subset V(C')$ consider one of its endpoints $y$, belonging to some $C_j$, say. Since $f_j$ is quasi-surjective, there is some $x \in V_j$ such that $d_{C}(f_j(x),y) \leq A$. Thus also $d_C(T_x',y)\leq A$ and so $d_{C'}(T_x',e)\leq A+1$. 
\end{proof}

\begin{remark} \label{rem TD}
   As mentioned in the introduction, \Tr{tree decomposition gen} becomes false if we replace $d_G$ by the intrinsic metric $d_{G[V_i]}$. To see this, let \g be obtained from $K_5$ by subdividing each edge into a path of length $N\in \N$. Consider a tree decomposition of \g with just two bags, one of which consists of a single edge $e$, and the other consists of $G -e$. Then both induced bags are planar (in their intrinsic metrics), the single adhesion set has diameter 1, but \g is not $(M,A)$-quasi-planar for any fixed $M,A$ independent of $N$.
\end{remark}

Graph-decompositions are a natural extension of tree-decompositions (see \cite{DJKK} for the precise definition).
By considering the the former instead of the latter, \Tr{tree decomposition gen} can be further generalized. We leave the details to the interested reader.

\section{Refining quasi-isometries}

Recall that a map $\varphi: V(H) \to V(G)$ between graphs (or metric spaces) is \defi{bi-Lipschitz}, if it satisfies \ref{quasiisom:1} with $A=0$. If $\varphi$ is in addition bijective, we call it a  \defi{bi-Lipschitz equivalence}, and say that $G,H$ are {bi-Lipschitz equivalent} in this case.
\smallskip

In this section we study the relationship between quasi-isometries and bi-Lipschitz equivalences. Some of our results will be used in the proof of \Tr{qi implies iter scig} in the next section, while others may be helpful elsewhere.

We start by identifying conditions under which a quasi-isometry can be made injective and/or surjective. For injectivity we have

\begin{proposition}[{\cite[Observation 2.2]{GeoPapMin}}] \label{injective}
    Suppose $f:G\to H$ is a $(M,A)$-quasi-isometry between graphs. Then there is an injective $(M',A')$-quasi-isometry $f'$ from $G$ into a graph obtained by attaching leaves (alternatively,  cliques) to the vertices of $H$.
    In addition, if $f$ is surjective then so is $f'$.
\end{proposition}

For surjectivity we have

\begin{proposition} \label{surjective}
    Suppose $f:G\to H$ is a $(M,A)$-quasi isometry between graphs. Then there is a surjective $(M',A')$-quasi-isometry $f'$ from $G$ onto an $A$-shallow contraction minor of $H$.
    In addition, if $f$ is injective then so is $f'$.
\end{proposition}

\begin{proof}
    Consider a Voronoi cell decomposition of $H$ with respect to $f(G)$. Note that any Voronoi cell has diameter at most $A$ by \ref{quasiisom:2}, and it is connected. Thus by contracting each Voronoi cell, we obtain a contraction minor of $H$ that is $(A+1,0)$-quasi-isometric to $H$.
\end{proof}

Next, applying the Voronoi decomposition to \g instead of $H$, we improve preimages of vertices under any quasi-isometry.

\begin{proposition} \label{connected preimage}
    Suppose $f:G\to H$ is a $(M,A)$-quasi-isometry between graphs. Then there is a $(M,A')$-quasi-isometry $f':G\to H$ \st\ \fe\ $y\in H$, the pre-image $f'^{-1}(y)$ is empty or connected. Moreover, if $f$ is surjective then so is $f'$.
\end{proposition}

\begin{proof}
    Consider a Voronoi cell decomposition of $G$ with respect to $\{ p_y \}_{y \in f(V(G))}$, where each $p_y$ is an arbitrary point in $f^{-1}(y)$.
    First note that for any $x \in V(G)$, we have
    \begin{equation} \label{Vor 1}
        d_G(x, p_{f(x)}) \leq AM.
    \end{equation}

    \comment{
    Also, we claim that \labtequ{Vor 2}{
    any Voronoi cell $R_y := R_{\{p_y\}}$ has diameter at most $2AM$.}
    Indeed for any points $x,x'$ belonging to a single Voronoi cell $R_y$, we have $d_G(x,x')\leq d_G(p_{y}, x) + d_G(p_{y}, x') \leq  d_G(p_{f(x)}, x) + d_G(p_{f(x')}, x') \leq 2AM$.
    }

    We define $f':G\to H$ by letting $f'(x) = y$ whenever $x \in R_y$.
    By construction, every pre-image $f'^{-1}(y) = R_y$ is a Voronoi cell and hence it is connected.
    With this definition, for every $x \in V(G)$, we have $x \in R_{f'(x)}$ and thus $d_G(x, p_{f'(x)}) \leq d_G(x, p_{y})$ for any $y \in f(V(G))$.
    Combining this with \eqref{Vor 1}, we obtain
    
    \begin{equation} \label{Vor 2}
        d_G(x, p_{f'(x)}) \leq AM.
    \end{equation}

    We claim that $f'$ is a quasi-isometry.
    For this, it is sufficient to show that $f$ and $f'$ differ point-wise by a uniform constant.
    Precisely, for every $x \in G$, we have that 
    $f(x) = f(p_{f(x)})$ and $f'(x) = f(p_{f'(x)})$, and hence $d_H(f(x),f'(x)) = d_H(f(p_{f(x)}),f(p_{f'(x)}))$.
    Since $f$ is a $(M,A)$-\qy, this quantity is at most $M d_G(p_{f(x)}, p_{f'(x)}) + A$.
    Applying the triangle inequality and \eqref{Vor 1} and \eqref{Vor 2}, we can bound the latter by\\ $M (d_G(x, p_{f(x)}) + d_G(x, p_{f'(x)})) + A \leq M(AM+AM)+A=2AM^2+A$.
    Since $f$ is a $(M,A)$-quasi-isometry, we deduce that $f'$ is a $(M, 2(2AM^2+A)+A)$-quasi-isometry.
\end{proof}

Let us summarize the above facts:

\begin{corollary} \label{leaves and sh contr}
     Let \cx\ be a class of graphs and $G$ be a graph in $\mathcal{QI}(\cx)$. Then 
     \begin{enumerate}
         \item \label{lshc i} if \cx\ is closed under adding leaves, there is an injective quasi-isometry from $G$ to an element of $\cx$;
         \item \label{lshc ii} if \cx\ is closed under shallow contraction minors, there is a surjective quasi-isometry from $G$ to an element of $\cx$, obtained by contracting connected subgraphs;
         \item \label{lshc iii} if \cx\ is closed under adding leaves and under shallow contraction minors, 
         then $G$ is bi-Lipschitz equivalent to an element of $\cx$.
     \end{enumerate}
\end{corollary}
\begin{proof}
    Item~\ref{lshc i} is just \Prr{injective}, and~\ref{lshc ii} is just \Prr{surjective}. Item~\ref{lshc iii} follows by applying both \Prr{injective} and \Prr{surjective} in either order. 
\end{proof}

\medskip
We remark that \Cr{leaves and sh contr} applies in particular to graphs embeddable into an arbitrary surface and that a similar fact was also noted in \cite[Lemma 2.3]{DaviesStringGraphsQuasiPlanar}.

\begin{corollary} \label{biL planar}
    Suppose \g is quasi-isometric to a planar graph $H$. Then \g is bi-Lipschitz equivalent to a planar graph $H'$. 
\end{corollary}

In general, we cannot let $H'=H$ here, as shown by Georgakopoulos \& Vigolo \cite{GeoVigCoa}.

\section{Cover intersection graphs and quasi-isometries}

In this section we prove \Tr{qi implies iter scig}, the main effort being the forward implication.
\medskip

Let $G,H$ be quasi-isometric graphs.
The following lemma will allow us to reduce to the case where $G,H$ are bi-Lipschitz equivalent.

\begin{lemma} \label{VG = VH}
    Let $G,H$ be $(M,A)$-quasi-isometric graphs. Then $H$ is bi-Lipschitz equivalent to a graph $G'$ which is a \cig\ of a subdivision of $G$, where the diameter of the sets of the cover depends only on $M,A$. 
\end{lemma}

Before proving this, we remark the following easy fact.

\begin{remark}\label{cig itself}
    Any graph \g is a \cig\ with subdivisions of itself. Indeed, by subdividing each edge of \g once, and taking the cover consisting of the stars centered at the original vertices of $G$, the resulting \cig\ is isomorphic to $G$.
\end{remark}

\begin{proof}[Proof of \Lr{VG = VH}]
    Let $f:H \rightarrow G$ be a $(M,A)$-quasi-isometry. By Proposition \ref{surjective} and Proposition \ref{injective} combined, $H$ is bi-Lipschitz equivalent to a graph $G'$ obtained by first taking a shallow contraction minor of $G$ and then attaching cliques to its vertices.
    We show that these operations can be realized using a single bounded-diameter $\SCIG$ operation on $G$. The required cover of $G$ is natural and can be obtained by readapting the one as in \Rr{cig itself}. To obtain $G'$ from $G$, cover a subdivision of the latter as in \Rr{cig itself}. Edge contractions can be obtained by merging sets of the cover, and clique attachments by adding new sets to the cover, consisting of a single vertex each, around {the original} vertices of $G$. Since the aforementioned contractions are shallow, our sets have diameter uniformly bounded by a function of $M,A$.
\end{proof}

We remark that such operations (contraction and clique attachment) can also be performed without using subdivisions in a manner similar to what will be done later (that also allows for leaf attachments).

\subsection{Edge-sliding} \label{sec ES}
We now introduce our \defi{edge-sliding}, a flexible operation which modifies a graph $G$ into a graph $H$ on the same vertex set that is bi-Lipschitz equivalent to $G$. 

Let $x,y,z$ be distinct vertices of $G$, and suppose $yz,zx\in E(G)$. We imagine `sliding' the endvertex $z$ of the edge $zx$ along the edge $yz$, while keeping the other endvertex $x$ fixed, to produce the edge $xy$. We thereby call $yz$ the  \defi{slider}, and $x$ the \defi{pivot}. Note that the distance between any two of $x,y,z$ is either 1 or 2 both before or after the operation, and this holds whether we retain or remove the sliding edge $zx$ after the operation. 

We would like to be able to perform an arbitrary number of such slidings simultaneously while keeping the resulting metric distortion bounded. In order for this to work, we need a `frame' of edges that are not allowed to move; these will be exactly the edges in $G\cap H$, and $yz$ above is an example of such an edge. 

Note that the sliding operation described above is invertible, and so rather than saying that $H$ is obtained from $G$, we could think of $G,H$ as `twins'. All this is summarized in the following definition,  defining a symmetric relation between graphs:


\begin{definition}
    Let $G, H$ be graphs. We say that $G,H$ are \defi{edge-sliding twins} if:
    \begin{enumerate}
        \item $V(G) = V(H)=:V$;
        \item for every $xy \in E(G)\sm E(H)$ there is $z \in V$ such that $yz \in E(G \cap H)$ and $zx \in E(H)$;
        \item for every $xy \in E(H) \sm E(G)$ there is $z \in V$ such that $yz \in E(G \cap H)$ and $zx \in E(G)$.
    \end{enumerate}
    In this case, we also say that $G$ is obtained from $H$ via an edge-sliding operation (and vice-versa) and we call $G \cap H$ the \defi{frame}.
\end{definition}

\begin{examp} \normalfont 
    The graphs $G$ and $H$ as in \fig{fig:edge-sliding} are edge-sliding twins.
    Indeed, suppose we want to construct $H$ from $G$. We fix $G \cap H$ (thick edges) as the frame and let:
    \begin{itemize}
        \item $\{2,4\}$ slide to $\{3,4\}$, with $4$ as the pivot and $\{2,3\}$ as the slider;
        \item a copy of $\{3,4\}$ slide to $\{3,5\}$, with $3$ as pivot and $\{4,5\}$ as slider ($\{3,4\}$ also remains as an edge of $H$);
        \item a copy of $\{1,6\}$ slide to
        $\{1,5\}$ (with $1$  as pivot and $\{5,6\}$ as slider) and another copy slide to $\{2,6\}$ (with $6$ as pivot and $\{1,2\}$ as slider).
    \end{itemize}
\end{examp}

\begin{figure}[h]
\centering
\begin{tikzpicture}[
    vertex/.style={circle, fill=black, inner sep=1.2pt},
    blackedge/.style={draw=black, thin},
    blackedge2/.style={draw=black, very thick},
]

\def\xone{-2.0}  \def\yone{-1}
\def\xtwo{0}  \def\ytwo{-1}
\def\xthree{2} \def\ythree{-1}
\def\xfour{2}  \def\yfour{1}
\def\xfive{0}  \def\yfive{1.0}
\def\xsix{-2}   \def\ysix{1.0}

\newcommand{\drawvertices}[1]{
    \node[vertex, label=below:1]  (1#1) at (\xone, \yone) {};
    \node[vertex, label=below:2] (2#1) at (\xtwo, \ytwo) {};
    \node[vertex, label=below:3] (3#1) at (\xthree, \ythree) {};
    \node[vertex, label=above:4] (4#1) at (\xfour, \yfour) {};
    \node[vertex, label=above:5] (5#1) at (\xfive, \yfive) {};
    \node[vertex, label=above:6] (6#1) at (\xsix, \ysix) {};
}

\newcommand{\drawblackedges}[1]{
    \draw[blackedge2] (1#1) -- (2#1) -- (3#1) -- (4#1) -- (5#1) -- (6#1);
    \draw[blackedge2] (2#1) -- (5#1);
}

\begin{scope}[scale=0.9]
    \node at (-3.0, -0.8) {$G$};
    
    \drawvertices{G}
    \drawblackedges{G}

    \draw[blackedge] (1G) -- (6G);
    \draw[blackedge] (2G) -- (4G);

\end{scope}

\begin{scope}[xshift=6cm, scale=0.9]
    \node at (-3.0, -0.8) {$H$};
    
    \drawvertices{H}
    \drawblackedges{H}
    
    \draw[blackedge] (1H) -- (5H);
    \draw[blackedge] (2H) -- (6H);
    \draw[blackedge] (3H) -- (5H);
    
\end{scope}
\end{tikzpicture}

\caption{An example of a pair of edge-sliding twins.}
\label{fig:edge-sliding}
\end{figure}
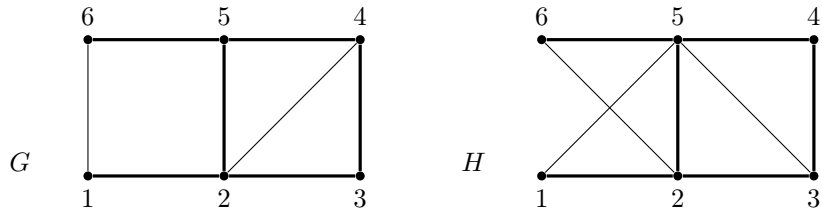

\begin{remark} \label{rem: est implies biL}
If two graphs $G,H$ are edge-sliding twins, then $G \subseteq H^2$ and $H \subseteq G^2$ and hence $G,H$ are $2$-bi-Lipschitz equivalent.
Moreover, if both $G,H$ are contained in $(G \cap H)^2$, then they are edge-sliding twins.
\end{remark}

The converse implications do not hold in general.
For example, $C_3 \vee C_3 \subseteq C_5^2$ and $C_5 \subseteq (C_3 \vee C_3)^2$, but $C_3 \vee C_3, C_5$ are not edge-sliding twins; also $K_{1,3}$ and $C_4$ are edge-sliding twins, but neither is contained in the square of their intersection, which is disconnected.

\begin{definition}
    Let $G, H$ be graphs. We say that $G$ and $H$ are {($k$-)}iterated edge-sliding twins, if there exist $k \in \N$ and a sequence of graphs $H=G_0, G_1, G_2, \ldots, G_k = G$ where $G_i$ and $G_{i-1}$ are edge-sliding twins for every $i \in [k]$. 
\end{definition}

\begin{lemma} \label{lem: iest}
Let $G$ be a connected graph and let $x, y, x', y' \in V(G)$.
Then $G \cup \{xy\}$ and $G \cup \{x'y'\}$ are {$(d_G(x,x')+d_G(y,y'))$-}iterated edge-sliding twins.
\end{lemma}

\begin{proof}
Let $k=d_G(x,x')$ and $k'=d_G(y,y')$, and consider two {geodesics} $P = (x = v_0, v_1, \ldots, v_k = x')$ and $Q = (y = w_0, w_1, \ldots, w_{k'} = y')$.
Then 
$$G \cup \{xy\}, G \cup \{v_1y\}, \ldots, G \cup \{v_{k-1}y\}, G \cup \{x'y\}, G \cup \{x'w_1\}, \ldots, G \cup \{x'y'\}$$
is a sequence {of length $k+k'+1$} in which consecutive graphs are edge-sliding twins, whereby each edge-sliding operation uses $G$ as the frame (and each edge in $P\cup Q$ is used as the slider exactly once
\end{proof}

In the setting of \Lr{lem: iest}, we call such a sequence from $G \cup \{xy\}$ to $G \cup \{x'y'\}$ a \defi{sliding of $xy$ to $x'y'$ (along $G$)}.

\begin{remark} \label{rem: est eventually}
    If two graphs $G, H$ are $k_0$-iterated edge-sliding twins, then they are $k$-iterated edge-sliding twins for every $k \geq k_0$.  In particular, any sliding of $xy$ to $x'y'$ of length $k_0$ along a connected graph $G$ can be extended to a sliding of length $k \geq k_0$.
\end{remark}

\begin{lemma} \label{lem: biL and est}
    Let $G,H$ be graphs. The following are equivalent:
    \begin{enumerate}
        \item \label{bil-est i} $G,H$ are bi-Lipschitz equivalent with constant $M \geq 1$;
        \item \label{bil-est ii} $G,H$ are $k$-iterated edge-sliding twins, for some $k \geq 0$;  
    \end{enumerate}
    Moreover, if $G,H$ satisfy the conditions above, then $k \leq 2 \lceil M \rceil $ and $M \leq 2^k$.
\end{lemma}

\begin{proof}
    To show that \ref{bil-est i} implies \ref{bil-est ii}, suppose $G$ and $H$ are bi-Lipschitz equivalent via the identity map.
    We claim that $G \cup H$ and $H$ are $\lceil M \rceil$-iterated edge-sliding twins. 
    Indeed, for each $e=uv \in E(G) \setminus E(H)$,
    we have that $d_H(u,v) \leq \lceil M \rceil d_G(u,v) = \lceil M \rceil$.
    For every such $e$, let $e' \in E(H)$ belong to a geodesic path in $H$ between $u,v$.
    By \Lr{lem: iest} and \Rr{rem: est eventually}, we have that $H \cup \{e'\}=H$ and $H \cup \{e\}$ are $\lceil M \rceil$-iterated edge-sliding twins, and there is a sliding of $e'$ to $e$ along $H$ of length $\lceil M \rceil+1$.
    We denote such a sliding with $\{G_{i,e}\}_{i=0}^{\lceil M \rceil}$, where $G_{i,e} = H \cup \{x_{i,e}y_{i,e}\}$.  
    We set $e' = x_{0,e}y_{0,e}$ and $e = x_{\lceil M \rceil,e}y_{\lceil M \rceil,e}$, so that $G_{0,e} = H$ and $G_{\lceil M \rceil,e} = H \cup \{e\}$.
    Then the sequence
    $G_{i} = \bigcup_{e \in E(G) \setminus E(H)} G_{i,e}$,
    where $G_{0} = H$ and $G_{\lceil M \rceil} = G \cup H$, witnesses that $H, G \cup H$ are $\lceil M \rceil$-iterated edge-sliding twins.

    By the same argument, $G \cup H$ and $G$ are $\lceil M \rceil$-iterated edge-sliding twins too.
    Recalling that being iterated edge-sliding twins is an equivalence relation, we conclude that $G$ and $H$ are $2\lceil M \rceil$-iterated edge-sliding twins.

    \medskip
    The implication \ref{bil-est ii} $\Rightarrow$ \ref{bil-est i} follows easily by recursively applying \Rr{rem: est implies biL}, once for each iteration.
\end{proof}

We end this subsection with the following crucial lemma about edge-sliding twins.

\begin{lemma} \label{lem: est implies cig}
Let $G$ and $H$ be edge-sliding twins. Then $H$ is a \scig\ of a \scig\ of \G, whereby the covers consist of subgraphs with diameter at most 4.
\end{lemma}

Before proving this, let us see how it completes the proof of \Tr{qi implies iter scig}.

\subsection{Proof of Theorem~\ref{qi implies iter scig}}

For the implication  \ref{qi-itercig ii} $\Rightarrow$ \ref{qi-itercig i}, recall that a bounded-diameter \cig\ of a graph $G$ is quasi-isometric to $G$ itself (\Cr{bounddiamsets}).
Hence any finite sequence of bounded-diameter $\SCIG$ operations (i.e.\ taking a \cig\ of a subdivision of a graph) preserves the quasi-isometry class.
In fact, by \Cr{bounddiamsets}, the quasi-isometry constants depend only on the number of iterations and the diameter of the sets of the cover.

\medskip

    For the implication \ref{qi-itercig i} $\Rightarrow$ \ref{qi-itercig ii},  by \Lr{VG = VH}, we have reduced to the case where two graphs $G,H$ are bi-Lipschitz equivalent or, by \Lr{lem: biL and est}, are iterated edge-sliding twins.
    In fact, it suffices to consider the case where $G,H$ are edge-sliding twins (not iterated), since the forward direction of \Lr{lem: biL and est} gives a bound on the number of iterations that depends only on the bi-Lipschitz equivalence constant $M$.
    If we assume that $G$ and $H$ are edge-sliding twins, then \Lr{lem: est implies cig} shows that $H$ can be obtained from $G$ using two $\SCIG$ operations with covers consisting of subgraphs with diameter at most 4.
    This concludes the proof. \qed

\smallskip
As mentioned in the introduction, this also completes the proof of \Tr{pl iff subd iter str}.

\subsection{Proof of \Lr{lem: est implies cig}}

We first observe some general facts about cover-intersection graphs and introduce some notation.

\begin{remark} \label{basiccase}
Every graph \g is a \cig\ of a \cig\ of itself, whereby the covers consist of subgraphs with diameter at most 1.
\end{remark}

We describe the covers used in this two-step process. For convenience, we assume \g is labelled and, abusing notation, we will not distinguish between a vertex and its label.
In the first step, \g is covered by its edges (\defi{edge-bags}) and its vertices (\defi{vertex-bags}). The resulting \cig\ $G'$ is a supergraph of the line graph of \g (which instead would result from a cover of \g consisting only of edge-bags). 
Edge-bags and vertex-bags are assigned one or two \defi{red labels}, corresponding to the vertices of \g which they enclose. The vertices of $G'$ will inherit the same red labels. By construction, vertices in $G'$ are joined by an edge if and only if they have a red label in common. Also, for every $v \in V(G)$, the vertices of $G'$ that bear $v$ as a red label induce a clique isomorphic to $K_{\deg(v)+1}$ in $G'$.
In the second step, we cover $G'$ with \defi{clique-bags} $\{K^v\}_{v \in V(G)}$, each enclosing one such clique. It is easy to check that the resulting \cig\ is isomorphic to $G$, by mapping every $v \in V(G)$ to $K^v$.

\begin{proof}[Proof of \Lr{lem: est implies cig}]    
    Starting from $G$, we will adapt the covers used in \Rr{basiccase} to obtain, after two steps, the desired edge-sliding twin $H$ of $G$: edges belonging only to $G$ will be ``absorbed'' during the process, and edges belonging only to $H$ will ``emerge'' from $G$.
    An adaptation of the procedure as in \Rr{basiccase} applied to $G \cap H$ will allow to simultaneously perform such absorption and emergence.
    
    Starting from $G=G_0$, we will construct the following sequence of labelled graphs: a subdivision $G_1$ of $G_0$; a \cig\ $G_2$ of $G_1$; a subdivision $G_3$ of $G_2$; and, finally, a \cig\ $G_4$ of $G_3$, chosen to be isomorphic to $H$.
    For each step $i$, we describe the construction of $G_i$ and its labeling:
\smallskip

    \textbf{Step 0}: To define our labeling of $G_0=G$, let $E_H = E(H) \setminus E(G)$ and $E_G = E(G) \setminus E(H)$. Recall that $V=V(G)=V(H)=V(G \cap H)$. By our assumption, there is $T: E_G \rightarrow E(H)$ that maps any edge $xy \in E_G$ to some $zx \in E(H)$ such that $yz \in E(G \cap H)$.
    Intuitively, the `redundant' edge $xy$ will be absorbed by $zx = T(xy)$, by keeping its pivot $x$ fixed and moving the other endpoint along the slider $yz$.
    Analogously, by our assumption there is $S: E_H \rightarrow E(G)$ that maps any edge $xy \in E_H$ to some $zx \in E(G)$ such that $yz \in E(G \cap H)$.
    Intuitively, the `desired' edge $xy$ will emerge from $zx = S(xy)$, by keeping its pivot $x$ fixed and moving the other endpoint along the slider $yz$.
    We attach \defi{green labels} to the edges in $E(G)$ that are involved in this absorption/emergence process: for every edge $e \in E_G$, attach $T(e)$ to it; for each $e \in E_H$, attach $e$ to $S(e)$ (\fig{fig: step 0}).
    Note that this in general leads to multiple labels being attached to an edge of $G$.
    As notation, any edge $xy$ attached to $G$ as a green label will be denoted with $(x,y)$.

    \begin{figure}[h!]
    \centering
    \begin{tikzpicture}[
        vertex/.style={circle, fill=black, inner sep=1.2pt},
        blackedge/.style={draw=black, thin},
        blackedge2/.style={draw=black, very thick},
        edgelabel/.style={fill=white, text=green!70!black, font=\scriptsize\bfseries, inner sep=1.5pt},
        redlabel/.style={text=red, font=\scriptsize\bfseries},
    ]
    
    \def\xone{-2.0}  \def\yone{-1}
    \def\xtwo{0}  \def\ytwo{-1}
    \def\xthree{2} \def\ythree{-1}
    \def\xfour{2}  \def\yfour{1}
    \def\xfive{0}  \def\yfive{1.0}
    \def\xsix{-2}   \def\ysix{1.0}
    
    \newcommand{\drawvertices}[1]{
        \node[vertex, label=below:1]  (1#1) at (\xone, \yone) {};
        \node[vertex, label=below:2] (2#1) at (\xtwo, \ytwo) {};
        \node[vertex, label=below:3] (3#1) at (\xthree, \ythree) {};
        \node[vertex, label=above:4] (4#1) at (\xfour, \yfour) {};
        \node[vertex, label=above:5] (5#1) at (\xfive, \yfive) {};
        \node[vertex, label=above:6] (6#1) at (\xsix, \ysix) {};
    }
    
    \newcommand{\drawblackedges}[1]{
        \draw[blackedge2] (1#1) -- (2#1) -- (3#1) -- (4#1) -- (5#1) -- (6#1);
        \draw[blackedge2] (2#1) -- (5#1);
    }
    
    \begin{scope}[scale=0.9]
        \node at (-3.2, -0.8) {$G=G_0$};
        
        \drawvertices{G}
        \drawblackedges{G}
    
        \draw[blackedge] (1G) -- node[edgelabel, pos=0.5, sloped] {(1,5),(2,6)} (6G);
        \draw[blackedge] (2G) -- node[edgelabel, pos=0.5, sloped] {(3,4)} (4G);
        
        \draw[blackedge] (3G) -- node[edgelabel, pos=0.5, sloped] {(3,5)} (4G);
        
    \end{scope}
    
    \begin{scope}[xshift=6cm, scale=0.9]
        \node at (-3.0, -0.8) {$H$};
        
        \drawvertices{H}
        \drawblackedges{H}
        
        \draw[blackedge] (1H) -- (5H);
        \draw[blackedge] (2H) -- (6H);
        \draw[blackedge] (3H) -- (5H);
        
    \end{scope}
    \end{tikzpicture}
    \caption{To construct $H$, we first attach green labels to some edges of $G$.}
    \label{fig: step 0}
    \end{figure}
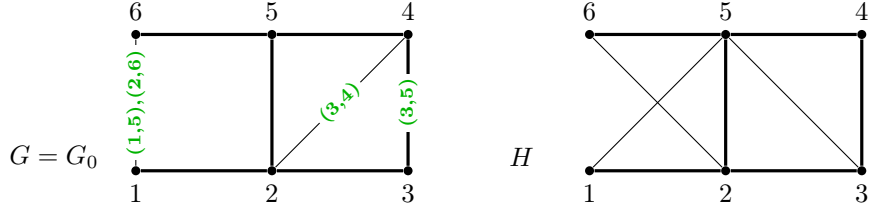

    We record the following rather obvious observation:
    \labtequ{cl green labels}{for every $x,y \in V$, we have $xy \in E(H)$ if and only if $xy \in E(G \cap H)$ or there is a green label $(x,y)$ in $G_0$.}
    Indeed, for the backward implication, note that green labels take values in $T(E_G) \cup E_H \subseteq E(H)$.
    For the forward implication, if $xy \in E(G \cap H)$ we are done. Otherwise, $xy \in E_H$ and, by construction, $S(xy)$ carries $(x,y)$ as a green label.

    \medskip

    \textbf{Step 1}: Construct $G_1$ by subdividing each edge of $G$ having at least one green label into a path on $3$ edges (\fig{fig: step 1}). Each $3$-path obtained by subdividing an edge $e$ inherits all the green labels of $e$.    
    Let $M$ be the subgraph of $G_1$  spanned by the edges of $G \cap H$ and any 3-paths that arose by subdividing edges of $G\cap H$. (Thus $M$ is a subdivision of $G \cap H$.)

    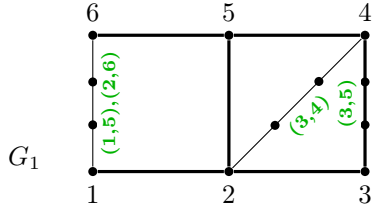
\begin{figure}[h!]
    \centering
    \begin{tikzpicture}[
        vertex/.style={circle, fill=black, inner sep=1.2pt},
        blackedge/.style={draw=black, thin},
        blackedge2/.style={draw=black, very thick},
        edgelabel/.style={fill=white, text=green!70!black, font=\scriptsize\bfseries, inner sep=1.5pt},
        redlabel/.style={text=red, font=\scriptsize\bfseries},
    ]
    
    \def\xone{-2.0}  \def\yone{-1}
    \def\xtwo{0}  \def\ytwo{-1}
    \def\xthree{2} \def\ythree{-1}
    \def\xfour{2}  \def\yfour{1}
    \def\xfive{0}  \def\yfive{1.0}
    \def\xsix{-2}   \def\ysix{1.0}
    
    \newcommand{\drawvertices}[1]{
        \node[vertex, label=below:1]  (1#1) at (\xone, \yone) {};
        \node[vertex, label=below:2] (2#1) at (\xtwo, \ytwo) {};
        \node[vertex, label=below:3] (3#1) at (\xthree, \ythree) {};
        \node[vertex, label=above:4] (4#1) at (\xfour, \yfour) {};
        \node[vertex, label=above:5] (5#1) at (\xfive, \yfive) {};
        \node[vertex, label=above:6] (6#1) at (\xsix, \ysix) {};
    }
    
    \newcommand{\drawblackedges}[1]{
        \draw[blackedge2] (1#1) -- (2#1) -- (3#1) -- (4#1) -- (5#1) -- (6#1);
        \draw[blackedge2] (2#1) -- (5#1);
    }
    
    \begin{scope}[yshift=-12cm, scale=0.9]
    
        \node at (-3.0, -0.8) {$G_1$};
        
        \drawvertices{G}
        \drawblackedges{G}
    
        \draw[blackedge] (1G) -- node[pos=0.33, vertex] {} node[pos=0.67, vertex] {} node[edgelabel, pos=0.5, sloped, below=1.5pt]{(1,5),(2,6)} (6G);
        \draw[blackedge] (2G) -- node[pos=0.33, vertex] {} node[pos=0.67, vertex] {} node[edgelabel, pos=0.5, sloped, below=1.5pt] {(3,4)} (4G);
        
        \draw[blackedge] (3G) -- node[pos=0.33, vertex] {} node[pos=0.67, vertex] {} node[edgelabel, pos=0.5, sloped, above=1.5pt] {(3,5)} (4G);
        
    \end{scope}
    \end{tikzpicture}
    \caption{$G_1$ is a subdivision of $G_0$}
    \label{fig: step 1}
    \end{figure}
    
    \medskip

    \textbf{Step 2}: Let $\cc_1$ be the cover of $M$ induced by the cover of $G \cap H$ given by edge-bags and vertex-bags (\fig{fig: step 2a} (left)), which, with a slight abuse, we will still call the \defi{edge-bags} and \defi{vertex-bags} of $\cc_1$. Let $A$ be the intersection graph of $\cc_1$ (\fig{fig: step 2a} (right)).

    \begin{figure}[h!]
    \centering
    
    \begin{tikzpicture}[
        vertex/.style={circle, fill=black, inner sep=1.2pt},
        blackedge/.style={draw=black, thin},
        blackedge2/.style={draw=black, very thick},
        edgelabel/.style={fill=white, text=green!70!black, font=\scriptsize\bfseries, inner sep=1.5pt},
        redlabel/.style={text=red, font=\scriptsize\bfseries},
        vertexbag/.style={
            circle, 
            draw=red, 
            thick, 
            fill=none,
            /utils/exec={\pgfgettransformentries{\myscale}{\tmp}{\tmp}{\tmp}{\tmp}{\tmp}},
            inner sep={\myscale*5pt}
        },
        edgebag/.style={
            to path={
                let \p1=($(\tikztotarget.center)-(\tikztostart.center)$),
                    \n1={veclen(\x1,\y1)} in
                \pgfextra{
                    \pgfgettransformentries{\myscale}{\tmp}{\tmp}{\tmp}{\tmp}{\tmp}
                }
                (\tikztostart.center) -- (\tikztotarget.center) 
                node[
                    midway, sloped, draw=red, thick, rounded rectangle, 
                    minimum width={\myscale*\n1 + \myscale*10pt}, 
                    minimum height={\myscale*10pt}, 
                    inner sep=0pt
                ] {}
                \tikztonodes
            }
        },
        ]
        
        \def\xone{-2.0}  \def\yone{-1}
        \def\xtwo{0}  \def\ytwo{-1}
        \def\xthree{2} \def\ythree{-1}
        \def\xfour{2}  \def\yfour{1}
        \def\xfive{0}  \def\yfive{1.0}
        \def\xsix{-2}   \def\ysix{1.0}
        
        \newcommand{\drawvertices}[1]{
            \node[vertex, label=below:1]  (1#1) at (\xone, \yone) {};
            \node[vertex, label=below:2] (2#1) at (\xtwo, \ytwo) {};
            \node[vertex, label=below:3] (3#1) at (\xthree, \ythree) {};
            \node[vertex, label=above:4] (4#1) at (\xfour, \yfour) {};
            \node[vertex, label=above:5] (5#1) at (\xfive, \yfive) {};
            \node[vertex, label=above:6] (6#1) at (\xsix, \ysix) {};
        }
        
        \newcommand{\drawblackedges}[1]{
            \draw[blackedge2] (1#1) -- (2#1) -- (3#1) -- (4#1) -- (5#1) -- (6#1);
            \draw[blackedge2] (2#1) -- (5#1);
        }
            
        \begin{scope}[scale=0.9]
        
            \node at (-3.0, -0.8) {$G_1$};
        
            \drawvertices{G}
            \drawblackedges{G}
        
            \draw[blackedge] (1G) -- node[pos=0.33, vertex] {} node[pos=0.67, vertex] {} node[edgelabel, pos=0.5, sloped, below=1.5pt]{(1,5),(2,6)} (6G);
            \draw[blackedge] (2G) -- node[pos=0.33, vertex] {} node[pos=0.67, vertex] {} node[edgelabel, pos=0.5, sloped, below=1.5pt] {(3,4)} (4G);
            
            \draw[blackedge] (3G) -- node[pos=0.33, vertex] {} node[pos=0.67, vertex] {} node[edgelabel, pos=0.5, sloped, above=1.5pt] {(3,5)} (4G);

            \node[vertexbag] at (1G) {};
            \node[vertexbag] at (2G) {};
            \node[vertexbag] at (3G) {};
            \node[vertexbag] at (4G) {};
            \node[vertexbag] at (5G) {};
            \node[vertexbag] at (6G) {};

            \draw (1G) to[edgebag] (2G);
            \draw (2G) to[edgebag] (3G);
            \draw (3G) to[edgebag] (4G);
            \draw (4G) to[edgebag] (5G);
            \draw (5G) to[edgebag] (6G);
            \draw (2G) to[edgebag] (5G);
            
        \end{scope}

        \begin{scope}[xshift=6cm, scale=0.4]
        \node at (-6, -3) {$A$};
        
        \node[vertex, label={[redlabel]0:2,5}] (25) at (0, 0) {};
        \node[vertex, label={[redlabel]90:5}]  (5)  at (0, 1.5) {};
        \node[vertex, label={[redlabel]270:2}]  (2)  at (0, -1.5) {};
        
        \node[vertex, label={[redlabel]90:5,6}]  (56) at (-2, 3) {};
        \node[vertex, label={[redlabel]90:4,5}]  (45) at (2, 3) {};
        \node[vertex, label={[redlabel]270:1,2}] (12) at (-2, -3) {};
        \node[vertex, label={[redlabel]270:2,3}] (23) at (2, -3) {};
        
        \node[vertex, label={[redlabel]135:6}] (6) at (-4, 1.5) {};
        \node[vertex, label={[redlabel]225:1}] (1) at (-4, -1.5) {};
        
        \node[vertex, label={[redlabel]90:4}]   (4)  at (4, 1.5) {};
        \node[vertex, label={[redlabel]270:3}]  (3)  at (4, -1.5) {};
        \node[vertex, label={[redlabel]0:3,4}] (34) at (4, 0) {};
        
        \draw[blackedge2] (56) -- (45);
        \draw[blackedge2] (56) -- (5);
        \draw[blackedge2] (45) -- (5);
        \draw[blackedge2] (56) -- (25);
        \draw[blackedge2] (45) -- (25);
        \draw[blackedge2] (5) -- (25);
        
        \draw[blackedge2] (12) -- (23);
        \draw[blackedge2] (12) -- (2);
        \draw[blackedge2] (23) -- (2);
        \draw[blackedge2] (12) -- (25);
        \draw[blackedge2] (23) -- (25);
        \draw[blackedge2] (2) -- (25);
        
        \draw[blackedge2] (45) -- (4);
        \draw[blackedge2] (45) -- (34);
        \draw[blackedge2] (4) -- (34);
        
        \draw[blackedge2] (23) -- (3);
        \draw[blackedge2] (23) -- (34);
        \draw[blackedge2] (3) -- (34);
        
        \draw[blackedge2] (6) -- (56);
        \draw[blackedge2] (1) -- (12);
        
        \end{scope}
    \end{tikzpicture}
    \caption{The subgraphs of $G_1$ in $\cc_1$ induce $A$ as \cig.}
    \label{fig: step 2a}

    \end{figure}
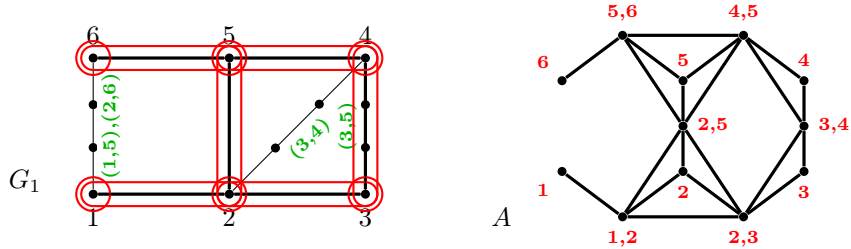

    We will extend some of the sets of $\cc_1$ to obtain a cover of $G_1$. The resulting \cig\ $G_2$ will be spanned by $A$. We describe how to extend the sets.
    As in \Rr{basiccase}, we assume that edge-bags and vertex-bags bear two or one \defi{red labels} respectively, namely the vertices they contain. For every green label $(x,y)$, attached to some $zx$-path $P$ of length $3$ where $yz \in E(G \cap H)$, proceed as follows: extend the vertex-bag labeled with the pivot $x$ to include an additional vertex of $P$, and extend the edge-bag labeled with the slider $yz$ to include two additional vertices of $P$ (\fig{fig: step 2b} (left)). We remark that for this extension it does not matter whether we slide $xy$ to $zx$ or the other way round. This process results in a cover of $G_1$ and we denote with $G_2$ the resulting \cig\ (\fig{fig: step 2b} (right)).
    
    \begin{figure}[h!]
    \centering
    \begin{tikzpicture}[
        vertex/.style={circle, fill=black, inner sep=1.2pt},
        blackedge/.style={draw=black, thin},
        blackedge2/.style={draw=black, very thick},
        edgelabel/.style={fill=white, text=green!70!black, font=\scriptsize\bfseries, inner sep=1.5pt},
        redlabel/.style={text=red, font=\scriptsize\bfseries},
        vertexbag/.style={
            circle, 
            draw=red, 
            thick, 
            fill=none,
            /utils/exec={\pgfgettransformentries{\myscale}{\tmp}{\tmp}{\tmp}{\tmp}{\tmp}},
            inner sep={\myscale*5pt}
        },
        edgebag/.style={
            to path={
                let \p1=($(\tikztotarget.center)-(\tikztostart.center)$),
                    \n1={veclen(\x1,\y1)} in
                \pgfextra{
                    \pgfgettransformentries{\myscale}{\tmp}{\tmp}{\tmp}{\tmp}{\tmp}
                }
                (\tikztostart.center) -- (\tikztotarget.center) 
                node[
                    midway, sloped, draw=red, thick, rounded rectangle, 
                    minimum width={\myscale*\n1 + \myscale*10pt}, 
                    minimum height={\myscale*10pt}, 
                    inner sep=0pt
                ] {}
                \tikztonodes
            }
        },
        vertexbag_enlarged/.style={
        to path={
            let \p1=($(\tikztotarget.center)-(\tikztostart.center)$),
                \n1={veclen(\x1,\y1)} in
            \pgfextra{
                \pgfgettransformentries{\myscale}{\tmp}{\tmp}{\tmp}{\tmp}{\tmp}
            }
            (\tikztostart.center) -- (\tikztotarget.center) 
            node[
                midway, sloped, draw=red, thick, rounded rectangle, 
                minimum width={\myscale*\n1 + \myscale*14pt}, 
                minimum height={\myscale*14pt}, 
                inner sep=0pt
            ] {}
            \tikztonodes
        }
    },
    ]
        
        \def\xone{-2.0}  \def\yone{-1}
        \def\xtwo{0}  \def\ytwo{-1}
        \def\xthree{2} \def\ythree{-1}
        \def\xfour{2}  \def\yfour{1}
        \def\xfive{0}  \def\yfive{1.0}
        \def\xsix{-2}   \def\ysix{1.0}
        
        \newcommand{\drawvertices}[1]{
            \node[vertex, label=below:1]  (1#1) at (\xone, \yone) {};
            \node[vertex, label=below:2] (2#1) at (\xtwo, \ytwo) {};
            \node[vertex, label=below:3] (3#1) at (\xthree, \ythree) {};
            \node[vertex, label=above:4] (4#1) at (\xfour, \yfour) {};
            \node[vertex, label=above:5] (5#1) at (\xfive, \yfive) {};
            \node[vertex, label=above:6] (6#1) at (\xsix, \ysix) {};
        }
        
        \newcommand{\drawblackedges}[1]{
            \draw[blackedge2] (1#1) -- (2#1) -- (3#1) -- (4#1) -- (5#1) -- (6#1);
            \draw[blackedge2] (2#1) -- (5#1);
        }
            
        \begin{scope}[scale=0.9]
        
            \node at (-3.0, -0.8) {$G_1$};
        
            \drawvertices{G}
            \drawblackedges{G}
        
            \draw[blackedge] (1G) -- node[pos=0.33, vertex] (1GA) {} node[pos=0.67, vertex] (1GB) {} node[edgelabel, pos=0.5, sloped, below=1.5pt] {(1,5),(2,6)} (6G);
            \draw[blackedge] (2G) -- node[pos=0.33, vertex] (2GA) {} node[pos=0.67, vertex] (2GB) {} node[edgelabel, pos=0.5, sloped, below=1.5pt] {(3,4)} (4G);
            
            \draw[blackedge] (3G) -- node[pos=0.33, vertex] (3GA) {} node[pos=0.67, vertex] (3GB) {} node[edgelabel, pos=0.5, sloped, above=1.5pt] {(3,5)} (4G);

            \node[vertexbag] at (2G) {};
            \node[vertexbag] at (5G) {};

            \draw (3G) to[edgebag] (4G);
            \draw (2G) to[edgebag] (5G);

            \draw (1G) to[vertexbag_enlarged] (1GA);
            \draw (6G) to[vertexbag_enlarged] (1GB);
            \draw (4G) to[vertexbag_enlarged] (2GB);
            \draw (3G) to[vertexbag_enlarged] (3GA);

            \def\bp{0.18}
            \draw[draw=red, thick, rounded corners=6pt] 
                ($(5G) + (\bp, \bp)$) --    
                ($(6G) + (-\bp, \bp)$) --   
                ($(1GA) + (-\bp, -\bp)$) -- 
                ($(1GA) + (\bp, -\bp)$) --  
                ($(6G) + (\bp, -\bp)$) --   
                ($(5G) + (\bp, -\bp)$) --   
                cycle;

            \draw[draw=red, thick, rounded corners=6pt] 
                ($(1GB) + (-\bp, \bp)$) --   
                ($(1G)  + (-\bp, -\bp)$) --  
                ($(2G)  + (\bp, -\bp)$) --  
                ($(2G)  + (\bp, \bp)$) --   
                ($(1G)  + (\bp, \bp)$) --    
                ($(1GB) + (\bp, \bp)$) --    
                cycle;
        
            \draw[draw=red, thick, rounded corners=6pt] 
                ($(5G) + (-\bp, \bp)$) --    %
                ($(4G) + (\bp, \bp)$) --   %
                ($(3GA) + (\bp, -\bp)$) -- %
                ($(3GA) + (-\bp, -\bp)$) -- 
                ($(4G) + (-\bp, -\bp)$) --   %
                ($(5G) + (-\bp, -\bp)$) --   %
                cycle;

            \def\eps{0.09}
            \draw[draw=red, thick, rounded corners=6pt]  
                ($(2G) + (-1.4142 /2 * \bp - \eps, 1.4142 /2  * \bp -\eps)$) --    %
                ($(2GB) + (0, 1.4142 * \bp)$) --    %
                ($(2GB) + (1.4142 * \bp, 0)$) --    %
                ($(2G) + (2.4142 * \bp, \bp)$) --    %
                ($(3G) + (\bp, \bp)$) --    %
                ($(3G) + (\bp, -\bp)$) --    %
                ($(2G) + (-\eps *1.4142, -\bp)$) --    %
                cycle;
                    
        \end{scope}
        
        \begin{scope}[xshift=6cm, scale=0.4]
        
        \node at (-6, -3) {$G_2$};
        
        \node[vertex, label={[redlabel]0:2,5}] (25) at (0, 0) {};
        \node[vertex, label={[redlabel]90:5}]  (5)  at (0, 1.5) {};
        \node[vertex, label={[redlabel]270:2}]  (2)  at (0, -1.5) {};
        
        \node[vertex, label={[redlabel]90:5,6}]  (56) at (-2, 3) {};
        \node[vertex, label={[redlabel]90:4,5}]  (45) at (2, 3) {};
        \node[vertex, label={[redlabel]270:1,2}] (12) at (-2, -3) {};
        \node[vertex, label={[redlabel]270:2,3}] (23) at (2, -3) {};
        
        \node[vertex, label={[redlabel]135:6}] (6) at (-4, 1.5) {};
        \node[vertex, label={[redlabel]225:1}] (1) at (-4, -1.5) {};
        
        \node[vertex, label={[redlabel]90:4}]   (4)  at (4, 1.5) {};
        \node[vertex, label={[redlabel]270:3}]  (3)  at (4, -1.5) {};
        \node[vertex, label={[redlabel]0:3,4}] (34) at (4, 0) {};
        
        \draw[blackedge2] (56) -- (45);
        \draw[blackedge2] (56) -- (5);
        \draw[blackedge2] (45) -- (5);
        \draw[blackedge2] (56) -- (25);
        \draw[blackedge2] (45) -- (25);
        \draw[blackedge2] (5) -- (25);
        
        \draw[blackedge2] (12) -- (23);
        \draw[blackedge2] (12) -- (2);
        \draw[blackedge2] (23) -- (2);
        \draw[blackedge2] (12) -- (25);
        \draw[blackedge2] (23) -- (25);
        \draw[blackedge2] (2) -- (25);
        
        \draw[blackedge2] (45) -- (4);
        \draw[blackedge2] (45) -- (34);
        \draw[blackedge2] (4) -- (34);
        
        \draw[blackedge2] (23) -- (3);
        \draw[blackedge2] (23) -- (34);
        \draw[blackedge2] (3) -- (34);
        
        \draw[blackedge2] (6) -- (56);
        \draw[blackedge2] (1) -- (12);
        
        
        \draw[blackedge] (6) -- node[edgelabel, pos=0.6, sloped, above=1pt] {(2,6)}(12);
        \draw[blackedge] (1) -- node[edgelabel, pos=0.6, sloped, below=1pt] {(1,5)}(56);
        \draw[blackedge] (12) -- node[edgelabel, pos=0.5, sloped, below=1pt] {(1,5)}(56);
        
        \draw[blackedge] (3) -- node[edgelabel, pos=0.6, sloped, below=1pt] {(3,5)}(45);
        \draw[blackedge] (4) -- node[edgelabel, pos=0.6, sloped, above=1pt] {(3,4)}(23);
        \end{scope}
    \end{tikzpicture}
    \caption{$G_1$ is fully covered and $G_2$ is the resulting \cig.}
    \label{fig: step 2b}
    \end{figure}
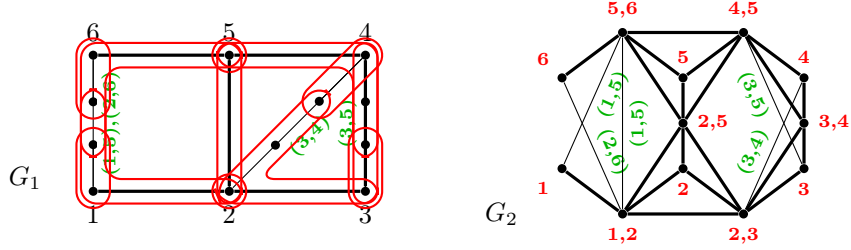
    
\medskip
    We now assign to each edge in $E(G_2) \setminus E(A)$ a single green label.
    While defining this assignment, we will verify that it satisfies the following two properties:
    \labtequ{cl G1 G2 gr lab first}{if there is a green label $(z,x)$ in $G_2$, then there is a green label $(z,x)$ in $G_1$,}
    and
    \labtequ{cl endpoints}{if an edge $e \in E(G_2)$ bears a green label $(z,x)$, one of its endpoints contains $z$ in its red label, and the other contains $x$.}
    Edges in $E(G_2) \setminus E(A)$ can be of two different types.
    The first type are those edges $e \in E(G_2) \setminus E(A)$ joining a pivot $x$ (vertex of $A$ with a single red label) and a slider $yz$ (vertex of $A$ with two red labels).
    Without loss of generality we can assume $xy \in E(G)$ and $zx \in E(H)$ and hence that $(z,x)$ is a green label in $G_0$ (attached to $xy$) also appearing in $G_1$.
    To such an edge $e$, we attach the green label $(z,x)$.
    As desired, properties \eqref{cl G1 G2 gr lab first} and \eqref{cl endpoints} are both satisfied.
    The second type are those edges $e \in E(G_2) \setminus E(A)$ joining two sliders $xy$ and $zt$, where $x,y,z,t \in V$ are distinct red labels.
    This happens exactly when some edge $f$ of $G$ slid along both $xy$ and $zt$. Suppose ---without loss of generality--- that $f$ contains $z$. Then the edge-bag $xy$ was enlarged to intersect the vertex-bag $z$ and there is an edge $e'$ in $E(G_2) \setminus E(A)$ also between $xy$ and $z$.
    We assign to $e$ the same green label as that of $e'$ (this will be either $(z,x)$ or $(y,z)$).
    Note that in both cases property \eqref{cl endpoints} is satisfied.
    Also, by the previous case, the green label attached to $e'$ satisfies \eqref{cl G1 G2 gr lab first}, hence also the same green label attached to $e$ does.

    This concludes the assignment of green labels to $G_2$. We can now strengthen \eqref{cl G1 G2 gr lab first} into
    \labtequ{cl G1 G2 gr lab}{there is a green label $(z,x)$ in $G_2$ if and only if there is a green label $(z,x)$ in $G_1$.}
    Indeed, the forward implication is \eqref{cl G1 G2 gr lab first}.
    For the other direction, note that the assignment of green labels to $G_2$ involves a choice, but every green label of $G_1$ will be assigned to some edge of $E(G_2)$ of the first type (this is the reason why we included vertex-bags in the cover $\cc_1$).
    More precisely, any green label $(z,x)$ appearing in $G_1$ is attached to some edge $xy \in G_0$ and, with this assignment, it is attached in $G_2$ to the edge with endpoints $x$ and $yz$.
    Note that it cannot happen that $(z,x)$ is attached to $xy$ in $G_0$ and at the same time $(x,y)$ is attached to $zx$, as it would imply that both $xy$ and $zx$ belong to $E(G \cap H)$.
    \medskip

    \textbf{Step 3}: We obtain $G_3$ from $G_2$ by subdividing every edge having a green label. Green and red labels pass on naturally to $G_3$ (green labels are now assigned to paths of length $2$; see \fig{fig: step 3}).
    Note that $A$ is a subgraph of $G_3$.
    \medskip

    \begin{figure}[h!]
        \centering
        \begin{tikzpicture}[
        vertex/.style={circle, fill=black, inner sep=1.2pt},
        blackedge/.style={draw=black, thin},
        blackedge2/.style={draw=black, very thick},
        edgelabel/.style={fill=white, text=green!70!black, font=\scriptsize\bfseries, inner sep=1.5pt},
        redlabel/.style={text=red, font=\scriptsize\bfseries},
        ]
            
        \begin{scope}[scale=0.4]
        
        \node at (-6, -3) {$G_3$};
        
        \node[vertex, label={[redlabel]0:2,5}] (25) at (0, 0) {};
        \node[vertex, label={[redlabel]90:5}]  (5)  at (0, 1.5) {};
        \node[vertex, label={[redlabel]270:2}]  (2)  at (0, -1.5) {};
        
        \node[vertex, label={[redlabel]90:5,6}]  (56) at (-2, 3) {};
        \node[vertex, label={[redlabel]90:4,5}]  (45) at (2, 3) {};
        \node[vertex, label={[redlabel]270:1,2}] (12) at (-2, -3) {};
        \node[vertex, label={[redlabel]270:2,3}] (23) at (2, -3) {};
        
        \node[vertex, label={[redlabel]135:6}] (6) at (-4, 1.5) {};
        \node[vertex, label={[redlabel]225:1}] (1) at (-4, -1.5) {};
        
        \node[vertex, label={[redlabel]90:4}]   (4)  at (4, 1.5) {};
        \node[vertex, label={[redlabel]270:3}]  (3)  at (4, -1.5) {};
        \node[vertex, label={[redlabel]0:3,4}] (34) at (4, 0) {};
        
        \draw[blackedge2] (56) -- (45);
        \draw[blackedge2] (56) -- (5);
        \draw[blackedge2] (45) -- (5);
        \draw[blackedge2] (56) -- (25);
        \draw[blackedge2] (45) -- (25);
        \draw[blackedge2] (5) -- (25);
        
        \draw[blackedge2] (12) -- (23);
        \draw[blackedge2] (12) -- (2);
        \draw[blackedge2] (23) -- (2);
        \draw[blackedge2] (12) -- (25);
        \draw[blackedge2] (23) -- (25);
        \draw[blackedge2] (2) -- (25);
        
        \draw[blackedge2] (45) -- (4);
        \draw[blackedge2] (45) -- (34);
        \draw[blackedge2] (4) -- (34);
        
        \draw[blackedge2] (23) -- (3);
        \draw[blackedge2] (23) -- (34);
        \draw[blackedge2] (3) -- (34);
        
        \draw[blackedge2] (6) -- (56);
        \draw[blackedge2] (1) -- (12);
        
        \draw[blackedge] (6) -- node[pos=0.5, vertex] {} node[edgelabel, pos=0.6, sloped, above=2pt] {(2,6)} (12);
        \draw[blackedge] (1) -- node[pos=0.5, vertex] {} node[edgelabel, pos=0.6, sloped, below=2pt] {(1,5)} (56);
        \draw[blackedge] (12) -- node[pos=0.5, vertex] {} node[edgelabel, pos=0.5, sloped, below=2pt] {(1,5)} (56);
        
        \draw[blackedge] (3) -- node[pos=0.5, vertex] {} node[edgelabel, pos=0.6, sloped, below=2pt] {(3,5)} (45);
        \draw[blackedge] (4) -- node[pos=0.5, vertex] {} node[edgelabel, pos=0.6, sloped, above=2pt] {(3,4)} (23);

        \end{scope}
        \end{tikzpicture}
        \caption{$G_3$ is a subdivision of $G_2$.}
        \label{fig: step 3}
    \end{figure}
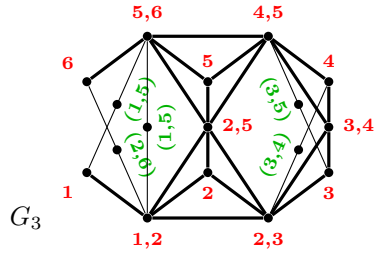

    \textbf{Step 4}: We now cover $A$ with clique-bags $\{K^v\}_{v \in V}$, where each $K^v$ is the clique induced by the vertices of $A$ bearing $v$ as a red label (\fig{fig: step 4} (left)). Let $B$ be the resulting intersection graph. As in \Rr{basiccase}, we have $B \isom G \cap H$, via the map that sends any $v \in V$ to $K^v$. We extend these clique-bags to cover $G_3$ with \defi{extended clique-bags $\{\bar K^v\}_{v \in V}$}: for every $2$-path $P$ in $G_3$ with a green label $(x,y)$, extend each of $K^x$ and $K^y$ to include the middle point of $P$ (\fig{fig: step 4} (right)).
    This operation is possible because, by \eqref{cl endpoints}, the clique-bags $K^x$ and $K^y$ include one endpoint of $P$ each, and so $\bar K^x, \bar K^y$ remain connected.
    This yields a cover of $G_3$ and we let $G_4$ be the resulting \cig.

    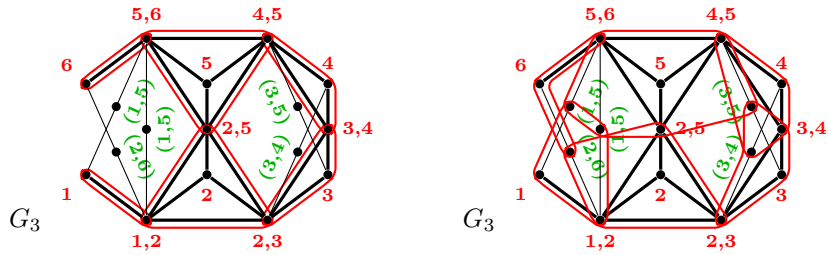
\begin{figure}[h!]
    \centering
    
    \begin{tikzpicture}[
        vertex/.style={circle, fill=black, inner sep=1.2pt},
        blackedge/.style={draw=black, thin},
        blackedge2/.style={draw=black, very thick},
        edgelabel/.style={fill=white, text=green!70!black, font=\scriptsize\bfseries, inner sep=1.5pt},
        redlabel/.style={text=red, font=\scriptsize\bfseries},
    ]  
        \begin{scope}[scale=0.4]
        \node at (-6, -3) {$G_3$};
        
        \node[vertex, label={[redlabel]0:2,5}] (25) at (0, 0) {};
        \node[vertex, label={[redlabel]90:5}]  (5)  at (0, 1.5) {};
        \node[vertex, label={[redlabel]270:2}]  (2)  at (0, -1.5) {};
        
        \node[vertex, label={[redlabel]90:5,6}]  (56) at (-2, 3) {};
        \node[vertex, label={[redlabel]90:4,5}]  (45) at (2, 3) {};
        \node[vertex, label={[redlabel]270:1,2}] (12) at (-2, -3) {};
        \node[vertex, label={[redlabel]270:2,3}] (23) at (2, -3) {};
        
        \node[vertex, label={[redlabel]135:6}] (6) at (-4, 1.5) {};
        \node[vertex, label={[redlabel]225:1}] (1) at (-4, -1.5) {};
        
        \node[vertex, label={[redlabel]90:4}]   (4)  at (4, 1.5) {};
        \node[vertex, label={[redlabel]270:3}]  (3)  at (4, -1.5) {};
        \node[vertex, label={[redlabel]0:3,4}] (34) at (4, 0) {};
        

        \draw[draw=red, thick, rounded corners=2.5pt]
            ($(1)  + (-0.25, 0)$) --
            ($(1)  + (0, 0.25)$) --
            ($(12) + (0, 0.25)$) --
            ($(12) + (0.25, 0)$) --
            ($(12) + (0, -0.25)$) --
            ($(1)  + (0, -0.25)$) -- cycle;
        
        \draw[draw=red, thick, rounded corners=2.5pt]
            ($(6)  + (-0.25, 0)$) --
            ($(6)  + (0, -0.25)$) --
            ($(56) + (0, -0.25)$) --
            ($(56) + (0.25, 0)$) --
            ($(56) + (0, 0.25)$) --
            ($(6)  + (0, 0.25)$) -- cycle;
        
        \draw[draw=red, thick, rounded corners=2.5pt]
            ($(25) + (0, -0.25)$) -- 
            ($(25) + (-0.25, 0)$) -- 
            ($(56) + (-0.25, 0)$) --
            ($(56) + (0, 0.25)$) --
            ($(45) + (0, 0.25)$) --
            ($(45) + (0.25, 0)$) --
            ($(25) + (0.25, 0)$) -- cycle;
        
        \draw[draw=red, thick, rounded corners=2.5pt]
            ($(25) + (-0.25, 0)$) --
            ($(12) + (-0.25, 0)$) --
            ($(12) + (0, -0.25)$) --
            ($(23) + (0, -0.25)$) --
            ($(23) + (0.25, 0)$) --
            ($(25) + (0.25, 0)$) --
            ($(25) + (0, 0.25)$) -- cycle;
        
        \draw[draw=red, thick, rounded corners=2.5pt]
            ($(45) + (0, 0.25)$) --
            ($(4)  + (0, 0.25)$) --
            ($(4)  + (0.25, 0)$) --
            ($(34) + (0.25, 0)$) --
            ($(34) + (0, -0.25)$) --
            ($(34) + (-0.25, 0)$) --
            ($(45) + (-0.25, 0)$) -- cycle;

        \draw[draw=red, thick, rounded corners=2.5pt]
            ($(23) + (0, -0.25)$) --
            ($(3)  + (0, -0.25)$) --
            ($(3)  + (0.25, 0)$) --
            ($(34) + (0.25, 0)$) --
            ($(34) + (0, 0.25)$) --
            ($(34) + (-0.25, 0)$) --
            ($(23) + (-0.25, 0)$) -- cycle;

        
        \draw[blackedge2] (56) -- (45);
        \draw[blackedge2] (56) -- (5);
        \draw[blackedge2] (45) -- (5);
        \draw[blackedge2] (56) -- (25);
        \draw[blackedge2] (45) -- (25);
        \draw[blackedge2] (5) -- (25);
        
        \draw[blackedge2] (12) -- (23);
        \draw[blackedge2] (12) -- (2);
        \draw[blackedge2] (23) -- (2);
        \draw[blackedge2] (12) -- (25);
        \draw[blackedge2] (23) -- (25);
        \draw[blackedge2] (2) -- (25);
        
        \draw[blackedge2] (45) -- (4);
        \draw[blackedge2] (45) -- (34);
        \draw[blackedge2] (4) -- (34);
        
        \draw[blackedge2] (23) -- (3);
        \draw[blackedge2] (23) -- (34);
        \draw[blackedge2] (3) -- (34);
        
        \draw[blackedge2] (6) -- (56);
        \draw[blackedge2] (1) -- (12);
        
        \draw[blackedge] (6) -- node[pos=0.5, vertex] {} node[edgelabel, pos=0.6, sloped, above=2pt] {(2,6)} (12);
        \draw[blackedge] (1) -- node[pos=0.5, vertex] {} node[edgelabel, pos=0.6, sloped, below=2pt] {(1,5)} (56);
        \draw[blackedge] (12) -- node[pos=0.5, vertex] {} node[edgelabel, pos=0.5, sloped, below=2pt] {(1,5)} (56);
        
        \draw[blackedge] (3) -- node[pos=0.5, vertex] {} node[edgelabel, pos=0.6, sloped, below=2pt] {(3,5)} (45);
        \draw[blackedge] (4) -- node[pos=0.5, vertex] {} node[edgelabel, pos=0.6, sloped, above=2pt] {(3,4)} (23);
        
        \end{scope}
        
        \begin{scope}[xshift=6cm, scale=0.4]
        \node at (-6, -3) {$G_3$};
        
        \node[vertex, label={[redlabel]0:2,5}] (25) at (0, 0) {};
        \node[vertex, label={[redlabel]90:5}]  (5)  at (0, 1.5) {};
        \node[vertex, label={[redlabel]270:2}]  (2)  at (0, -1.5) {};
        
        \node[vertex, label={[redlabel]90:5,6}]  (56) at (-2, 3) {};
        \node[vertex, label={[redlabel]90:4,5}]  (45) at (2, 3) {};
        \node[vertex, label={[redlabel]270:1,2}] (12) at (-2, -3) {};
        \node[vertex, label={[redlabel]270:2,3}] (23) at (2, -3) {};
        
        \node[vertex, label={[redlabel]135:6}] (6) at (-4, 1.5) {};
        \node[vertex, label={[redlabel]225:1}] (1) at (-4, -1.5) {};
        
        \node[vertex, label={[redlabel]90:4}]   (4)  at (4, 1.5) {};
        \node[vertex, label={[redlabel]270:3}]  (3)  at (4, -1.5) {};
        \node[vertex, label={[redlabel]0:3,4}] (34) at (4, 0) {};

        
        \draw[blackedge2] (56) -- (45);
        \draw[blackedge2] (56) -- (5);
        \draw[blackedge2] (45) -- (5);
        \draw[blackedge2] (56) -- (25);
        \draw[blackedge2] (45) -- (25);
        \draw[blackedge2] (5) -- (25);
        
        \draw[blackedge2] (12) -- (23);
        \draw[blackedge2] (12) -- (2);
        \draw[blackedge2] (23) -- (2);
        \draw[blackedge2] (12) -- (25);
        \draw[blackedge2] (23) -- (25);
        \draw[blackedge2] (2) -- (25);
        
        \draw[blackedge2] (45) -- (4);
        \draw[blackedge2] (45) -- (34);
        \draw[blackedge2] (4) -- (34);
        
        \draw[blackedge2] (23) -- (3);
        \draw[blackedge2] (23) -- (34);
        \draw[blackedge2] (3) -- (34);
        
        \draw[blackedge2] (6) -- (56);
        \draw[blackedge2] (1) -- (12);
        
        \draw[blackedge] (6) -- node[pos=0.5, vertex] (6A) {} node[edgelabel, pos=0.6, sloped, above=2pt] {(2,6)} (12);
        \draw[blackedge] (1) -- node[pos=0.5, vertex] (1A) {} node[edgelabel, pos=0.6, sloped, below=2pt] {(1,5)} (56);
        \draw[blackedge] (12) -- node[pos=0.5, vertex] (12A) {} node[edgelabel, pos=0.5, sloped, below=2pt] {(1,5)} (56);
        
        \draw[blackedge] (3) -- node[pos=0.5, vertex] (3A) {} node[edgelabel, pos=0.6, sloped, below=2pt] {(3,5)} (45);
        \draw[blackedge] (4) -- node[pos=0.5, vertex] (4A) {} node[edgelabel, pos=0.6, sloped, above=2pt] {(3,4)} (23);
        
        
        \draw[draw=red, thick, rounded corners=2.5pt]
            ($(1)  + (-0.25, 0)$) --
            ($(1A)  + (-0.25, 0)$) --
            ($(1A)  + (0, 0.25)$) --
            ($(12A)  + (0, 0.25)$) --
            ($(12A)  + (0.25, 0)$) --
            ($(12) + (0.25, 0)$) --
            ($(12) + (0, -0.25)$) --
            ($(1)  + (0, -0.25)$) -- cycle;
        
        \draw[draw=red, thick, rounded corners=2.5pt]
            ($(6)  + (-0.25, 0)$) --
            ($(6A)  + (-0.25, 0)$) --
            ($(6A)  + (0, -0.25)$) --
            ($(6A)  + (0.25, 0)$) --
            ($(6)  + (0.25, 0)$) --
            ($(56) + (0.25, 0)$) --
            ($(56) + (0, 0.25)$) --
            ($(6)  + (0, 0.25)$) -- cycle;
        
        \draw[draw=red, thick, rounded corners=2.5pt]
            ($(12A) + (0, -0.25)$) --
            ($(1A) + (0, -0.25)$) --
            ($(1A) + (-0.25, 0)$) --
            ($(56) + (-0.25, 0)$) --
            ($(56) + (0, 0.25)$) --
            ($(45) + (0, 0.25)$) --
            ($(45) + (0.25, 0)$) --
            ($(3A) + (0.25, 0)$) --
            ($(3A) + (0, -0.25)$) --
            ($(25) + (0, -0.25)$) -- cycle;
        
        \draw[draw=red, thick, rounded corners=2.5pt]
            ($(6A) + (0, 0.25)$) --
            ($(6A) + (-0.25, 0)$) --
            ($(12) + (-0.25, 0)$) --
            ($(12) + (0, -0.25)$) --
            ($(23) + (0, -0.25)$) --
            ($(23) + (0.25, 0)$) --
            ($(25) + (0.25, 0)$) --
            ($(25) + (0, 0.25)$) -- cycle;
        
        \draw[draw=red, thick, rounded corners=2.5pt]
            ($(45) + (0, 0.25)$) --
            ($(4)  + (0, 0.25)$) --
            ($(4)  + (0.25, 0)$) --
            ($(34) + (0.25, 0)$) --
            ($(34) + (0, -0.25)$) --
            ($(4A) + (0, -0.25)$) --
            ($(4A) + (-0.25, 0)$) --
            ($(45) + (-0.25, 0)$) -- cycle;
        
        \draw[draw=red, thick, rounded corners=2.5pt]
            ($(23) + (0, -0.25)$) --
            ($(3)  + (0, -0.25)$) --
            ($(3)  + (0.25, 0)$) --
            ($(34) + (0.25, 0)$) --
            ($(34) + (0, 0.25)$) --
            ($(3A) + (0, 0.25)$) --
            ($(3A) + (-0.25, 0)$) --
            ($(4A) + (-0.25, 0)$) --
            ($(23) + (-0.25, 0)$) -- cycle;

        \end{scope}

    \end{tikzpicture}
    \caption{clique-bags are extended to cover $G_3$ in full.}
    \label{fig: step 4}

    \end{figure}

    Note that $V(G_4) \isom V$, via the map that sends every vertex $v \in V$ to $\bar K^v$. 
    We claim that this map is in fact an isomorphism from $H$ to $G_4$:
    \labtequ{cl isom}{For every $x,y \in V$, we have $\bar K^x \bar K^y \in E(G_4)$ if and only if $xy \in E(G \cap H)$ or there is a green label $(x,y)$ in $G_0$.}

    Note that \eqref{cl isom} and \eqref{cl green labels} combined imply $G_4 \isom H$ as desired, hence it remains to prove \eqref{cl isom}.
    For the backward implication, if $xy \in E(G \cap H)$, then $K^x K^y \in E(B)$ and so $ \bar K^x \bar K^y \in E(G_4)$.
    Otherwise, suppose there is an edge in $G_0$ with a green label $(x,y)$.
    Then, by \eqref{cl G1 G2 gr lab}, there is also a green label $(x,y)$ on some $2$-path in $G_3$. Let $m$ be the midpoint of such a path. By construction, the clique-bags $K^x$ and $K^y$ have been enlarged to include $m$, so $m \in \bar K^x \cap \bar K^y$ and so $\bar K^x \bar K^y \in E(G_4)$.
    
    For the forward implication, suppose two enlarged clique-bags $\bar K^x, \bar K^y$ have non-empty intersection.
    By construction, the intersection can only contain  vertices from $V(G_3) \setminus V(A)$ or the vertex of $G_3$ bearing $xy$ as red label. If $\bar K^x \cap \bar K^y$ contains the vertex with red label $xy$, then $K^x K^y \in E(B)$ and so $xy \in E(G \cap H)$.
    If $\bar K^x \cap \bar K^y$ contains a vertex $m \in V(G_3) \setminus V(A)$, then $m$ is the midpoint of a $2$-path bearing $(x,y)$ as its green label and hence $(x,y)$ was a green label in $G_0$ by \eqref{cl G1 G2 gr lab} as desired.
\end{proof}

\subsection{Proof of \Cr{cayley}}

    We now adapt the methods of the above proof to the setting of groups acting on graphs.
    Let $G,H$ be graphs with the same vertex set $V$.
    If a group $\Gamma$ acts by isomorphisms on both $G$ and $H$, inducing the same action on $V$, we say that $\Gamma$ \defi{acts on $(G,H)$}.

    Note that if a group $\Gamma$ acts on a pair of edge-sliding twins $(G,H)$, any edge-sliding operation that transforms $G$ into $H$ is invariant under the action of $\Gamma$, namely every edge $xy \in E(G)$ slides to $zx \in E(H)$ if and only if, for every $g \in \Gamma$, the edge $(g \cdot x) (g \cdot y) \in E(G)$ slides to $(g \cdot z) (g \cdot x) \in E(H)$.
    In this case, we say that \defi{$\Gamma$ acts on an edge-sliding from $G$ to $H$}. 

    \begin{lemma} \label{lem: action on est}
        Let $\Gamma$ be a group that acts on an edge-sliding from $G$ to $H$. Then $G$ is an iterated \scig\ of $H$, using at each step a cover that is $\Gamma$-invariant.
    \end{lemma}

    \begin{proof}
        We adapt the four-step construction from the proof of \Tr{qi implies iter scig}.
        Since the graph covers of $G_1, G_3$ are determined by the green labels on $G_0, G_2$ respectively, it suffices to ensure that the assignments of green labels at Steps 0 and 2 are $\Gamma$-invariant.
        In Step 0, recall that we label some edges of $G_0$ using the maps $T,S$.
        By our assumption, the edge-sliding from $G$ to $H$ is $\Gamma$-invariant.
        Hence we can ask that, for every $g \in \Gamma$ and every edge $e$, we have $T(g \cdot e) = g \cdot T(e)$ and $S(g \cdot e) = g \cdot S(e)$.
        With this choice of $T,S$, the labeling of $G_0$ is $\Gamma$-invariant, and so is the resulting cover of $G_1$.
        Hence $\Gamma$ also acts by automorphisms on the intersection graph $G_2$.
        In Step 2, recall that we assign to every edge in $e \in E(G_2) \setminus E(A)$ a single green label, let us say $R(e)$.
        We claim that we can choose $R$ such that, for every edge $e \in E(G_2) \setminus E(A)$,
        \begin{equation} \label{eq: gamma inv}
        R( g \cdot e) = g \cdot R(e).
        \end{equation}
        Indeed, if $e\in E(G_2) \setminus E(A)$ joins a slider $yz$ with a pivot $x$, without loss of generality, we assume $xy \in E(G)$ is sliding to $zx \in E(H)$ and hence $R(e)=(z,x)$.
        Pick some $g \in \Gamma$; by our assumption there is an edge $g \cdot e \in E(G_2) \setminus E(A)$ between the slider $(g \cdot y)(g \cdot z)$ and the pivot $(g \cdot x)$.
        In this case, the assignment $R( g \cdot e) = (g \cdot z)(g \cdot x)$ is forced and satisfies \eqref{eq: gamma inv}.
        If an edge $e$ joins two sliders $xy$, $zt$, recall that we assign to $e$ the same label as of some edge $e'$ ---without loss of generality--- between $xy$ and $z$.
        Pick some $g \in \Gamma$; by our assumption there is 
        an edge $g \cdot e$ between $(g \cdot x)(g \cdot y)$ and $(g \cdot z)(g \cdot t)$ and an edge $g \cdot e'$ between $(g \cdot x)(g \cdot y)$ and $(g \cdot z)$.
        Hence a valid choice for the green label $R(g \cdot e)$ is $R(g \cdot e')$, which by the previous case is equal to $g \cdot R(e') = g \cdot R(e)$.
        Finally, the cover of $G_3$ is determined by the green labels attached to $G_2$ and, with this choice, it is $\Gamma$-invariant.
    \end{proof}

    \begin{lemma} \label{lem: action on biL}
        Let $\Gamma$ be a group that acts by isomorphisms on $(G,H)$, where $G,H$ are bi-Lipschitz equivalent graphs. 
        Then $G, H$ are iterated edge-sliding twins, where at each step $\Gamma$ acts on the corresponding edge-sliding operation.
    \end{lemma}

    \begin{proof}
        We adapt the proof of the implication \ref{bil-est i} $\Rightarrow$ \ref{bil-est ii} of \Lr{lem: biL and est}, ensuring that every choice is $\Gamma$-invariant.
        Suppose we want to show the statement for $G \cup H$ and $H$.
        Note that by our assumption $\Gamma$ acts on $G$ and $H$ with the same induced action on their vertex set, and hence it also acts on $(G \cup H,H)$.
        Recall that for any $e \in E(G) \setminus E(H)$, we chose a sliding $\{G_{i,e}\}^{\lceil M \rceil}_{i=0}$ from an edge $e' \in E(H)$ to $e$ of length $\lceil M \rceil+1$.
        At every step $i$, we had $G_{i,e} = H \cup \{x_{i,e}y_{i,e}\}$, for some vertices $x_{i,e}, y_{i,e}$.
        Using the additional hypothesis that $\Gamma$ acts on $(G \cup H,H)$ we can ask that, for every $g \in \Gamma$ and every $e \in E(G) \setminus E(H)$, we have $x_{i, g \cdot e} = g \cdot x_{i,e}$ and $y_{i, g \cdot e} = g \cdot y_{i,e}$.
        The resulting sequence of graphs $\{G_i\}_{i=0}^{\lceil M \rceil}$, obtained as in \Lr{lem: biL and est} is such that at every step the group $\Gamma$ acts on the corresponding edge-sliding operation.
    \end{proof}

    We are now ready to prove the main result of this subsection, \Cr{cayley}, which we restate for convenience:

    \begin{nocorollary} 
    Let $\Gamma$ be a finitely generated group and let $\Gamma' \leq \Gamma$ be a finite index subgroup. Let $S, S'$ be generating sets of $\Gamma, \Gamma'$ respectively. Then the Cayley graph $G = Cay(\Gamma, S)$ is an iterated \scig\ of $G'= Cay(\Gamma',S')$, using at each step a cover that is $\Gamma'$-invariant.
\end{nocorollary}

    \begin{proof}

    We first consider the case where $\Gamma=\Gamma'$.  
    Easily, in this case, the Cayley graphs $G, G'$ are bi-Lipschitz equivalent and $\Gamma$ acts on $(G, G')$.
    By \Lr{lem: action on biL} and \Lr{lem: action on est}, the statement follows.
    
    In the general case where $k:=[\Gamma : \Gamma'] < \infty$, we can apply the previous case to replace $S$ with any generating set of $G$.
    Hence we can assume $S = S' \cup T$, where $T = \{t_1, \dots, t_k\}$ is a set of representatives for the right cosets of $\Gamma'$ in $\Gamma$ and $t_1=\Id$.
    Let $G''$ be the graph obtained by attaching to each vertex $g \in G'$ a $k$-clique $K^g$.
    We associate each vertex of $K^g$ with the element $gt_i \in \Gamma$ (in particular, we associate $g$ with $gt_1=g$).
    The resulting graph $G''$ has vertex set $\Gamma$ and, for every $g, g'\in \Gamma'$ and $t, t' \in T$,
    \labtequ{cl action isom}{$(gt)(g't')$ is an edge of $G''$ if and only if $g = g'$ or $t = t' = \Id$ and $g'^{-1}g \in S'$.}
    
    Clearly $\Gamma'$ acts by left multiplication on the vertex set $\Gamma$ of $G''$ and we claim that this is indeed an action by isomorphisms on $G''$.
    Indeed, pick $h, g, g' \in \Gamma'$ and $t, t' \in T$.
    Then using \eqref{cl action isom}, there is an edge $(hgt)(hg't')$ in $G''$ if and only if $hg = hg'$ or $t = t' = \Id$ and $(hg')^{-1}hg \in S'$. 
    Using \eqref{cl action isom} again, the latter is equivalent to the existence of an edge $(gt)(g't')$, as desired.
    In fact, since $\Gamma'$ acts by isomorphisms (by left multiplication) also on $G$, we deduce that $\Gamma'$ acts on $(G'', G)$.
    
    If we can show that $G'', G$ are bi-Lipschitz equivalent, then, applying \Lr{lem: action on biL} and \Lr{lem: action on est} as above, we deduce that $G$ is an iterated \cig\ with subdivisions of $G''$, using at each step a cover that is $\Gamma'$-invariant.
    Furthermore, recalling that the construction of $G''$ can be realized via a \cig\ with subdivisions of $G'$ (as in the proof of \Lr{VG = VH}) using a cover that is $\Gamma'$-invariant, this would conclude the proof.

    It remains to check that $G'', G$ are bi-Lipschitz equivalent.
    Note that $G''$ is bi-Lipschitz equivalent to the graph $G^*$ obtained from $G'$ by attaching $k-1$ leaves to each vertex (instead of a clique).
    So it suffices to show that $G^*$ and $G$ are bi-Lipschitz equivalent, where $G^* \subseteq G$ since $S = S' \cup T$.
    Clearly, we have $d_G \leq d_{G^*}$.
    On the other hand, pick $g, g' \in \Gamma'$ and $t, t' \in T$.
    We have $d_{G^*}(gt, g't') \leq d_{G^*}(g, g') + 2 = d_{G'}(g, g') + 2 $. 
    Since $[\Gamma : \Gamma']$ is finite, it is well-known ---and easy to see--- that the inclusion map from $G'$ to $G$ is a $(M,A)$-quasi-isometry (for some $M \geq 1, A \geq 0$).
    Hence we can bound $d_{G'}(g, g') + 2$ by $M(d_G(g, g'))+A+2$.
    Recalling that $T \subseteq S$, we also have $d_G(g, gt) \leq 1$ and $d_G(g', g't') \leq 1$, hence $M(d_G(g, g'))+A+2 \leq M(d_G(gt, g't')+2)+A+2$.
    This proves the other desired inequality, and establishes that the identity map between $G^*$ and $G$ is a bijective quasi-isometry and hence the two graphs are bi-Lipschitz equivalent.    
    \end{proof}

\section{Open problems} \label{sec OP}

As mentioned in the introduction, we do not know whether subdivisions are necessary in \Tr{pl iff subd iter str}:
\begin{problem} \label{prob subd}
    Is every countable quasi-planar graph an iterated string graph?\\ In particular, is every finite graph an iterated string graph?
\end{problem} 

    \comment{We start with a potential strengthening of \Tr{Davies}, which may be too good to be true: 

    \begin{problem} \label{pr D gen}
    Let \cc\ be a contraction-closed class of graphs. Must every \cig\ of a graph of \cc\ be (uniformly) \qic\ to a graph in \cc?\\ (In other words, is $\CIG(\cc) \subseteq \QI_{(M,A)}(\cc)$ for fixed $M,A$ depending on \cc\ only)
\end{problem}

    Similarly, one can ask if $\QI(\cc)$ is contraction-closed \fe\ contraction-closed \cc; by our \Tr{cig to cm}, this would imply a positive answer to \Prb{pr D gen}. We could ask for even more: is it true \fe\ $G\in \cc$, and every \cig\ $H$ of $G$, that $H$ is \qic\ to a contraction minor of $G$?
}

\subsection{Forbidding induced minors}

It is well-known, and not hard to see, that if a graph $G$ forbids a graph $H$ as a minor, then no graph in $\CIG(G)$ contains the 1-subdivision $H^{(1)}$ of $H$ as an induced minor \cite{LeeSep,BonHickInd}. A question, independently posed  by Lokshtanov, McCarty and Wiederrecht,  asked about the converse, and Bonnet \& Hickingbotham \cite{BonHickInd} provided an example showing that it is generally not true that every graph forbidding some $H$ as an induced minor is a \cig\ of a graph in $Forb(K_t)$ for any fixed $t$. Our work raises the question of whether we can obtain a converse if we allow iterated \cig s, possibly mixed with bounded subdivisions:

\begin{question} \label{Q Find CIG}
    Let $H$ be a finite graph. Are there $t,n\in \N$ \st\ $\Find(H)\subseteq \CIG^n(\Forb(K_t))$? Or $\Find(H)\subseteq \SCIG^n(\Forb(K_t))$? 
\end{question}

Interestingly, a positive answer to this question would imply, by our \Cr{equivalence} and the fact that $\Forb(K_t)$ has finite asymptotic dimension \cite{LiuAss}, the following conjecture of Abrishami et al.: 
\begin{conjecture}[{\cite[Conjecture 8.6.]{ABDDMRW}}]  \label{conj ABDDMRW}
     For every finite graph $H$, the class $\Find(H)$ has finite asymptotic dimension.
\end{conjecture}

This in turn has interesting consequences on graph colouring, because as proved in the aforementioned paper \cite[Theorem 1.4]{ABDDMRW}, every hereditary graph class with finite asymptotic dimension is Burling-controlled. 

More generally, to prove \Cnr{conj ABDDMRW}, it would suffice to prove weaker versions of \Qr{Q Find CIG} where apart from bounded edge-subdivisions we could also use any other graph transformation that preserves the quasi-isometry class. See also \cite{sphereDimSoCG} for related questions about $\Find(H)$.

\subsection{A notion of stretch} \label{sec stretch}

Given a graph \G, and a spanning subgraph $H \subseteq G$, we define the \defi{$H$-stretch} of \g to be the smallest $K \in \N \cup \{\infty\}$ \st\ for every $xy \in E(G)$ we have $d_H(x,y) \leq K$. In the rich graph spanner literature one also says that $H$ is a $K$-spanner of $G$.

\begin{observation} \label{obs str}
Whenever $H\subseteq G$ is spanning, $G$ has bounded $H$-stretch if and only if the identity $\Id: V(H) \to V(G)$ is a quasi-isometry, and in fact a bi-Lipchitz equivalence.
\end{observation}

It is natural to ask whether this is enough to understand quasi-planar graphs, which was our initial motivation for studying this notion of stretch. In a similar spirit, Berger \& Seymour \cite{BerSeyBou} considered the stretch of (spanning) trees of \defi{quasi-trees}, i.e.\ graphs quasi-isometric to trees. Motivated by an analogue of \Or{obs str}, they wondered whether every $(M,A)$-quasi-tree \g admits a spanning tree $T$ with bounded $T$-stretch. They proved that the answer is negative, with $G$ being the Farey graph as an example \cite[Theorem 6.1]{BerSeyBou}. Nevertheless, will show that the answer becomes positive if we are allowed to raise $G$ to a power $K$ depending on $M,A$:

\begin{corollary} \label{cor power}
    For every $M\geq 1, A\geq 0$ there are $K,K'\in \N$ \st\ if \g is a $(M,A)$-quasi-tree, then $G^K$ has a spanning tree $T$ with $T$-stretch at most $K'$.
\end{corollary}

This holds much more generally, in particular when \g is quasi-planar, and we want a spanning planar sugraph. We will derive it from the following:

\begin{proposition} \label{cor: qi and bs}
    Let $\mathcal{C}$ be a class of graphs. Then for any graph $G$ the following are equivalent:
    \begin{enumerate}
        \item \label{qibs i} $G$ is bi-Lipschitz equivalent to an element of $\mathcal{C}$;
        \item \label{qibs ii} there exist some $H \in \mathcal{C}$ and $K \in \N$ such that $H^K$ has bounded $G$-stretch;
        \item \label{qibs iii} there exist some $H \in \mathcal{C}$ and $K \in \N$ such that $G^K$ has bounded $H$-stretch.
        \item \label{qibs iv} there exist some $H \in \mathcal{C}$ and $K \in \N$ such that $H^K$ has bounded $G$-stretch and $G^K$ has bounded $H$-stretch.
    \end{enumerate}
\end{proposition}

\begin{proof}
    It suffices to prove that \ref{qibs i} implies \ref{qibs iv}, since the latter entails both \ref{qibs ii}, \ref{qibs iii}, any of which implies \ref{qibs i} by \Or{obs str}. 
    
    So consider some  bi-Lipchitz equivalence $f$ from $G$ to some $H \in \cc$. Let us first check that $G^{K'}$ has bounded $H$-stretch for some $K'$. Indeed, since bi-Lipchitz equivalences are bijective, we may assume that $V(G)=V(H)$ and $f=\Id_{V(G)}$. Then any edge $xy \in E(H) \setminus E(G)$ satisfies $d_G(x,y) \leq K'$ for some $K'$, and so $H\subseteq G^{K'}$. Since $G^{K'}$ is quasi-isometric to $G$ it is also quasi-isometric to $H$, and so $G^{K'}$ has bounded $H$-stretch by \Or{obs str}. 
    
    We can repeat the argument on $f^{-1}$ to deduce that $H^{K''}$ has bounded $G$-stretch for some $K''
    \in \N$, and choose $K:= \max \{K',K''\}$.
\end{proof}

To see how this implies \Cr{cor power}, first apply part \ref{lshc iii} of \Cr{leaves and sh contr} to the class $\cx$ of trees, to deduce that any quasi-tree $G$ is bi-Lipschitz equivalent to a tree.
Then part \ref{qibs ii} of \Prr{cor: qi and bs} guarantees the 
existence of a tree $T'$ spanning $G^K$ with stretch at most $K'$, for some $K, K'$.

\smallskip
The aforementioned result of Berger \& Seymour motivates the following, and again we expect the answer to be negative.

\begin{question} \label{spanning planar}
    Let \g be a quasi-planar graph. Is there a spanning planar subgraph $H\subseteq G$ such that $G$ has bounded $H$-stretch?
\end{question}

Catusse, Chepoi \& Vax{\`e}s \cite{CCVHop} proved that the answer is positive for unit-disc graphs, which are a subclass of string graphs. 

\smallskip
The notion of stretch raises many algorithmic questions; we offer the following as a sample. 

\begin{question} \label{algo GH}
Is there a polynomial-time algorithm that given two graphs $G,H$ with $|V(G)|=|V(H)|$ computes/approximates\footnote{up to a fixed multiplicative constant} the $H$-stretch of \G?
\end{question}

Given a class \cp\ of graphs, we define the \defi{\cp-stretch}
of a graph $G$ as the minimal $H$-stretch of \g over all  ${H\in \cp}$. For example, we can talk about the planar stretch of \G.

\begin{question} \label{algo planar}
Is there a polynomial-time algorithm for computing/approximating the planar stretch of a finite graph?
\end{question}

Results on analogous problems \cite{ADGK2t,BDLLTre,CDNRV} suggest that exact computation may be NP-hard, but approximation up to a large enough multiplicative constant may be tractable.


\comment{
\begin{conjecture} \label{spanning planar}
    A graph \g is quasi-planar, \iff\ there is a spanning planar subgraph $H\subseteq G$ and $K\in \R$ \st\ $d_H(x,y)< K$ \fe\ $xy\in E(G) \sm E(H)$.
\end{conjecture}
}

\bibliographystyle{plain}
\bibliography{collective}

@string{combi = {Combinatorica}}

@string{ejc = {Europ.\ J.\ Comb.}}

@string{tams = {Trans.\ Am.\ Math.\ Soc.}}

@book{DiestelBook25,
	Author = {Diestel, Reinhard},
	Note = {\\ Electronic edition available at:\\ {\small\tt http://www.math.uni-hamburg.de/home/diestel/books/graph.theory}},
	Publisher = {Springer-Verlag},
	Title = {Graph {T}heory \emph{(6th edition)}},
	Year = {2025}}

@book{GR01,
	Author = {C.~Godsil and G.~Royle},
	Publisher = {Springer-Verlag},
	Title = {Algebraic Graph Theory},
	Year = {2001}}

@book{LyndonSchupp,
	Author = {Roger~C.~Lyndon  and Schupp, Paul E.},
	Isbn = {9783540411581},
	Language = {en},
	Month = jan,
	Publisher = {Springer Science \& Business Media},
	Title = {Combinatorial {Group} {Theory}},
	Year = {2001}}

@incollection{GroAsyInv,
	Author = {Gromov, M.},
	Title = {{Asymptotic invariants of infinite groups}},
	Booktitle = {{Geometric group theory, Vol.~2 (Sussex, 1991)}},
	Number = {182},
	Pages = {1--295},
	Publisher = {Camb.\ Univ.~Press},
	Series = {London Math.~Soc.~Lecture Note Ser.},
	Year = {1993},
}

@article{FujPapAsy,
	Author = {K.~Fujiwara and P.~Papasoglu},
	Title = {Asymptotic dimension of planes and planar graphs},
	volume = {374},
	journal = tams,
	year = {2021},
	pages = {8887--8901}
	}

@article{FujPapCoa,
	Author = {K.~Fujiwara and P.~Papasoglu},
	Title = {A COARSE-GEOMETRY CHARACTERIZATION OF CACTI},
	note = {{arXiv:2305.08512}},
	}

@article{GNRS,
	Author = {A.~Gupta and I.~Newman and Y.~Rabinovich and A.~Sinclair},
	Pages = {233--269},
	volume = {2},
	Title = {{Cuts, trees and $\ell_1$-embeddings of graphs}},
	journal = combi,
	Year = {2004},
}

@article{BBEGLPS,
	title = {Asymptotic {Dimension} of {Minor}-{Closed} {Families} and {Assouad}-{Nagata} {Dimension} of {Surfaces}},
	author = {Bonamy, M. and Bousquet, N. and Esperet, L. and Groenland, C. and Liu, C.-H. and Pirot, F. and Scott, A.},
	volume = {26},
	number = {10},
	journal = {J.\ Eur.\ Math.\ Soc.},
	year = {2023},
	pages = {3739--3791},
}

@article{GeoPapMin,
	title = {{Graph minors and metric spaces}},
	author = {A.~Georgakopoulos and P.~Papasoglu},
    journal = {Combinatorica},
    volume = {45},
    pages = {33},
	year = {2025}
}

@article{BerSeyBou,
	title = {Bounded diameter tree-decompositions},
	Journal = combi,
	author = {Berger, E. and Seymour, P.},
	volume = {44},
	number = {1},
	year = {2024},
	pages = {659--674}
}

@article{CDNRV,
	title = {Constant {Approximation} {Algorithms} for {Embedding} {Graph} {Metrics} into {Trees} and {Outerplanar} {Graphs}},
	volume = {47},
	number = {1},
	journal = {Discrete \& Computational Geometry},
	author = {Chepoi, V. and Dragan, F. F. and Newman, I. and Rabinovich, Y. and Vax\`es, Y.},
	year = {2012},
	pages = {187--214},
}

@article{LiuAss,
	title = {Assouad-{Nagata} dimension of minor-closed metrics},
	author = {Liu, C.-H.},
	Note = {{arXiv:2308.12273}},
}

@article{EsGiCoa,
	title = {Coarse geometry of quasi-transitive graphs beyond planarity},
	author = {Esperet, L. and Giocanti, U.},
	journal = ejc,
	volume = {31},
    number = 2,
    Pages = {P2.41},
	Year = {2024}
}

@article{DHIM,
	title = {Fat minors cannot be thinned (by quasi-isometries)},
	author = {Davies, J. and Hickingbotham, R. and Illingworth, F. and McCarty, R.},
    year = {2026},
    pages = {},
    volume = {14},
    journal = {Analysis and Geometry in Metric Spaces},
    number={1}
}

@article{McMFat,
	title = {Fat minors in finitely presented groups},
	author = {MacManus, J.},
    journal={Combinatorica},
    year = {2025},
    pages = {},
    volume = {45},
    number = {40},
}

@article{McMAcc,
	title = {{Accessibility, planar graphs, and quasi-isometries}},
	author = {MacManus, J.},
	note = {arXiv:2310.15242},
}

@article{AJKW,
    title = {A Characterisation of Graphs Quasi-isometric to ${K}_4$-minor-free Graphs},
	author = {Albrechtsen, S. and Jacobs, R. and Knappe, P. and Wollan, P.},
	year = {2025},
    pages = {},
    volume = {45},
    number = {61},
    journal = {Combinatorica},
}

@article{DJKK,
	title = {Canonical graph decompositions via coverings},
	author = {Diestel, R. and Jacobs, R.~W. and Knappe, P. and Kurkofka, J.},
	note = {arXiv:2207.04855},
}

@article{NgScSeAsyI,
	title = {{Asymptotic structure. I. Coarse tree-width}},
	author = {Nguyen, T. and Scott, A. and Seymour, P.},
	journal = {arXiv:2501.09839},
}

@article{NgScSeAsyII,
	title = {{Asymptotic structure. II. Path-width and additive quasi-isometry}},
	author = {Nguyen, T. and Scott, A. and Seymour, P.},
	Note = {Preprint 2024},
}

@article{ADGSmallCounterexamples,
    author = {Albrechtsen, S. and Distel, M. and Georgakopoulos, A.},
    title = {Small counterexamples to the fat minor conjecture},
    note = {arXiv:2601.05761},
}

@article{ADGK2t,
    author = {Albrechtsen, S. and Distel, M. and Georgakopoulos, A.},
    title = {{Excluding $K_{2,t}$ as a fat minor}},
    note = {arXiv:2510.14644},
}

@InProceedings{sphereDimSoCG,
  author =	{Davies, J. and Georgakopoulos, A. and Hatzel, M. and McCarty, R.},
  title =	{{Strongly Sublinear Separators and Bounded Asymptotic Dimension for Sphere Intersection Graphs}},
  booktitle =	{41st International Symposium on Computational Geometry (SoCG 2025)},
  pages =	{36:1--36:16},
  series =	{Leibniz International Proceedings in Informatics (LIPIcs)},
  year =	{2025},
  volume =	{332},
  editor =	{Aichholzer, O. and Wang, H.},
  address =	{Dagstuhl, Germany},
}

@article{DaviesStringGraphsQuasiPlanar,
	title = {String graphs are quasi-isometric to planar graphs},
	author = {Davies, J.},
	note = {arXiv:2510.19602},
}

@article{CCTZstring,
      title={{O(1)-Distortion Planar Emulators for String Graphs}}, 
      author={H.-C.~Chang and J.~Conroy and Z.~Tan and D.~W.~Zheng},
      note = {arXiv:2510.21700},
}

@article{BonHickInd,
  title = {Induced Minors and Region Intersection Graphs},
  author = {Bonnet, {\'E}. and Hickingbotham, R.},
  journal = {Innovations in Graph Theory},
  volume = {2},
  pages = {313--327},
  year = {2025},
}

@InProceedings{ABDDMRW,
    author = {Abrishami, T. and Bria{\'n}ski, M. and Davies, J. and Du, X. and Masa{\v{r}}{\'i}kov{\'a}, J. and Rz{\k{a}}{\.z}ewski, P. and Walczak, B.},
    title = {Burling Graphs in Graphs with Large Chromatic Number},
    booktitle = {Proceedings of the 2026 Annual ACM-SIAM Symposium on Discrete Algorithms (SODA)},
    pages = {3978-3998},
    year={2026}
}

@article{GeoVigCoa,
AUTHOR = {A.~Georgakopoulos and F.~Vigolo},
TITLE = {Triangulating surfaces quasi-isometrically},
	Note = {arXiv:2603.21189},
}

@InProceedings{LeeSep,
  author =	{J.~R.~Lee},
  title =	{{Separators in region intersection graphs}},
  booktitle =	{Proc.\ 8th Innovations in Theoretical Computer Science},
  pages =	{1:1--1:8},
  series =	{LIPIcs},
  year =	{2017},
  volume =	{67},
  editor =	{C.~H.~Papadimitriou},
  publisher =	{Schloss Dagstuhl},
}

@unpublished{AlbGeoBlo,
    author = {Albrechtsen, S.  and Georgakopoulos, A.},
    title = {A coarse block-cutvertex tree-decomposition},
    note = {arXiv:2607.07030},
}

@article {StrDualPath,
    author = {Albrechtsen, S. and Diestel, R. and Elm, A.-K. and Fluck, E. and Jacobs, R. W. and Knappe, P. and Wollan, P.},
    title = {A structural duality for path-decompositions into parts of small radius},
    journal = {Innovations in Graph Theory},
    volume = {3},
    pages = {207--246},
    year = {2026},
}

@misc{BLPPSep,
      title={{Coarse Balanced Separators in Fat-Minor-Free Graphs}}, 
      author={{\'E}.~Bonnet and H.~Le and Ma.~Pilipczuk and Mi.~Pilipczuk},
      note = {arXiv:2604.11318}, 
}

@article{LiuWeakCoarseMenger,
    author = {C.-H.~Liu},
    title = {{Coarse Menger property of quasi-minor excluded graphs and length spaces}},
    note = {arXiv:2605.10068}
}

@misc{DGHLMGallai,
      title={{A coarse Gallai theorem}}, 
      author={M.~Distel and U.~Giocanti and J.~Hodor and C.~Legrand-Duchesne and P.~Micek},
      note={arXiv:2601.18439},
      }

@article{NgScSeAsyVI,
	title = {{Asymptotic structure.\ VI.\ Distant paths across a disc}},
	author = {Nguyen, T. and Scott, A. and Seymour, P.},
	Note = {ArXiv:2509.07174},
}

@article{BlPiPrCoa,
      title={{A coarse Menger's Theorem for planar and bounded genus graphs}}, 
      author={V.~Blazej and M.~Pilipczuk and E.~Protopapas},
      note={arXiv:2605.11112}, 
}

@InProceedings{CCVHop,
    author="Catusse, N.
    and Chepoi, V.
    and Vax{\`e}s, Y.",
    editor="Scheideler, C.",
    title="Planar Hop Spanners for Unit Disk Graphs",
    booktitle="Algorithms for Sensor Systems",
    year="2010",
    publisher="Springer Berlin Heidelberg",
    pages="16--30",
}

@book{AlgDes,
    author = {Kleinberg, J. and 
    Tardos, E.},
    title = {Algorithm Design},
    publisher = {Pearson},
    year = {2005}
}

@article{MaMaSpAlm,
      title={Almost planar finitely presented groups}, 
      author={J.~M.~Mackay and J.~P.~MacManus and D.~Spriano},
    Note={arXiv:2605.03040}      
}

@article{BDLLTre,
title = {Tree spanners on chordal graphs: complexity and algorithms},
journal = {Theor.\ Comput.\ Sci.},
volume = {310},
number = {1},
pages = {329--354},
year = {2004},
author = {A.~Brandst\"adt and F.~F.~Dragan and H.-O.~Le and V.~B.~Le},
}

\end{document}